\documentclass[12pt]{article}

\usepackage{amssymb}
\usepackage[cal=boondox]{mathalfa}
\usepackage{amsthm}
\usepackage{amsmath}
\usepackage{graphicx}
\usepackage{fullpage}
\usepackage{color}
\usepackage{enumerate}
 \numberwithin{equation}{section}
\usepackage{booktabs}
\allowdisplaybreaks

\usepackage{hyperref}
 \usepackage{mathtools}
 \mathtoolsset{showonlyrefs}
\usepackage[nocompress]{cite}

\theoremstyle{plain}
\newtheorem{thm}{Theorem}[section]
\newtheorem{cor}[thm]{Corollary}

\newtheorem{lem}[thm]{Lemma}
\newtheorem{prop}[thm]{Proposition}

\theoremstyle{definition}

\theoremstyle{remark}

\newtheorem{rem}[thm]{Remark}

\newcommand{\N}{\mathbb{N}}
\newcommand{\R}{\mathbb{R}}

\newcommand{\bp}{\begin{proof}[\ensuremath{\mathbf{Proof}}]}
\newcommand{\bs}{\begin{proof}[\ensuremath{\mathbf{Solution}}]}
\newcommand{\ep}{\end{proof}}
\newcommand{\be}{\begin{equation}}
\newcommand{\ee}{\end{equation}}

\begin{document}

\title{An eradication time problem for the SIR model}

\author{Ryan Hynd\footnote{Department of Mathematics, University of Pennsylvania. }\;, Dennis Ikpe\footnote{Department of Statistics and Probability, Michigan State University. African Institute for Mathematical Sciences, South Africa.}\;, and Terrance Pendleton\footnote{Department of Mathematics, Drake University.}\;}

\maketitle 

\begin{abstract} 
We consider a susceptible, infected, and recovered infectious disease model which incorporates a vaccination rate. In particular, we study the problem of choosing the vaccination rate in order to reduce the number of infected individuals to a given threshold as quickly as possible. This is naturally a problem of time-optimal control.  We interpret the optimal time as a solution of two dynamic programming equations and give necessary and sufficient conditions for a vaccination rate to be optimal.
\end{abstract}



\section{Introduction}
The SIR infectious disease model in epidemiology involves the system of ODE 
\be\label{SIR}
\begin{cases}
\dot S=-\beta S I\\
\dot I=\beta SI -\gamma I\\
\dot R=\gamma I.
\end{cases}
\ee 
Here $S,I,R: [0,\infty)\rightarrow \R$ represent the susceptible, infected, and recovered compartments of a total population, and $\beta>0$ and $\gamma>0$ are the respective infected and recovery rates per unit time.  It is also clear that once $S, I$ are determined then $R$ is known. As a result, we only need to consider the first two equations.

\begin{figure}[h]
\centering
 \includegraphics[width=.9\textwidth]{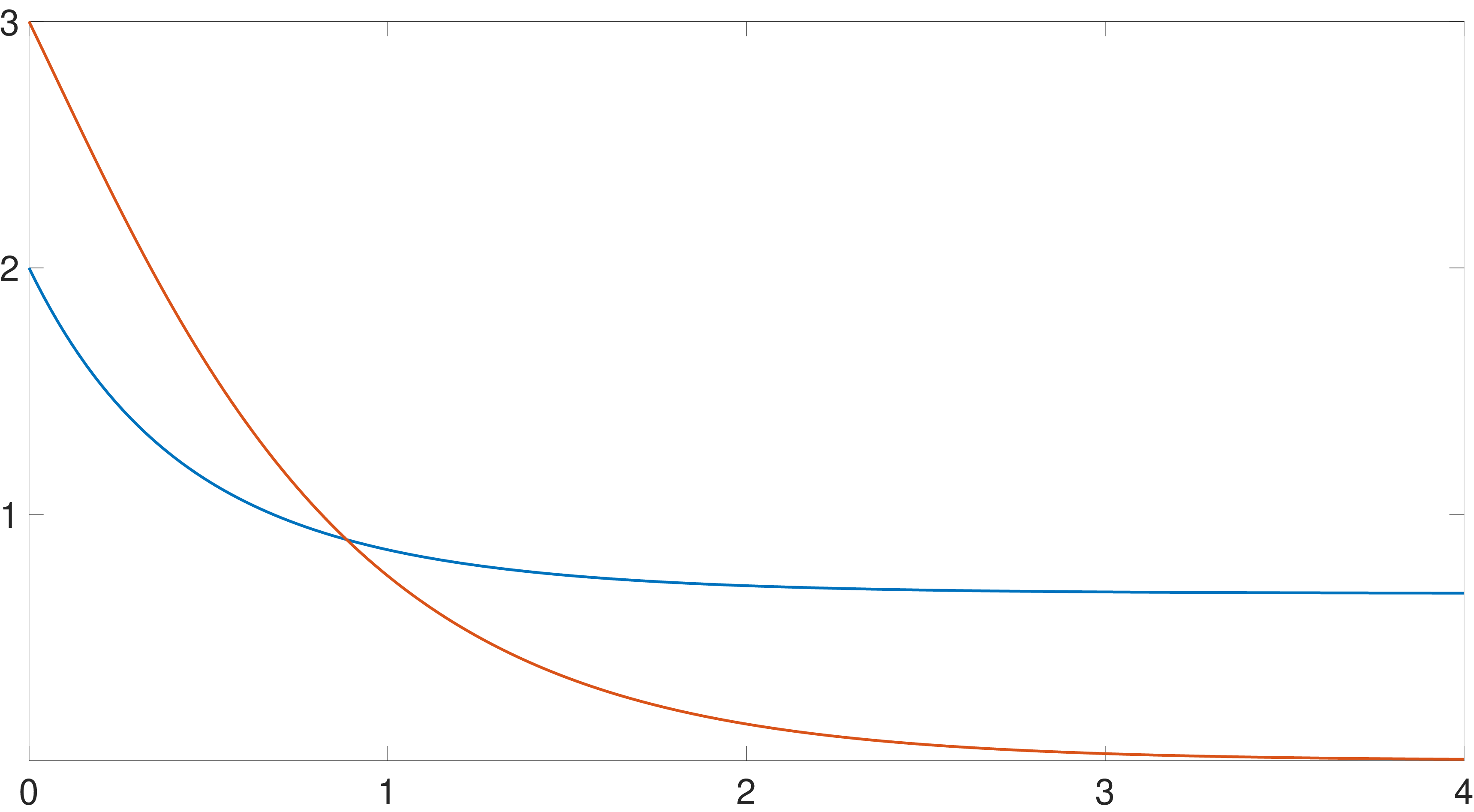}
 \caption{Solution of the SIR system with $S(0)=2$, $I(0)=3$, $\beta=1/2$, and $\gamma=2$. The graph of $S$ is shown in blue, and the graph of $I$ is shown in red. Note that $\beta S(0)\le \gamma$ so that $I$ is decreasing.}\label{SIdecrease}
\end{figure}

\par It is not hard to see that any solution $S, I$ of \eqref{SIR} with initial conditions $S(0), I(0)>0$, remains positive and bounded with $S$ decreasing. Moreover, if
$$
\beta S(0)\le \gamma,
$$
then $I$ is also decreasing. Otherwise, $I$ increases for an interval of time and decreases from then on. And in either case, 
$$
\lim_{t\rightarrow\infty}I(t)=0.
$$ 
\begin{figure}[h]
\centering
 \includegraphics[width=.9\textwidth]{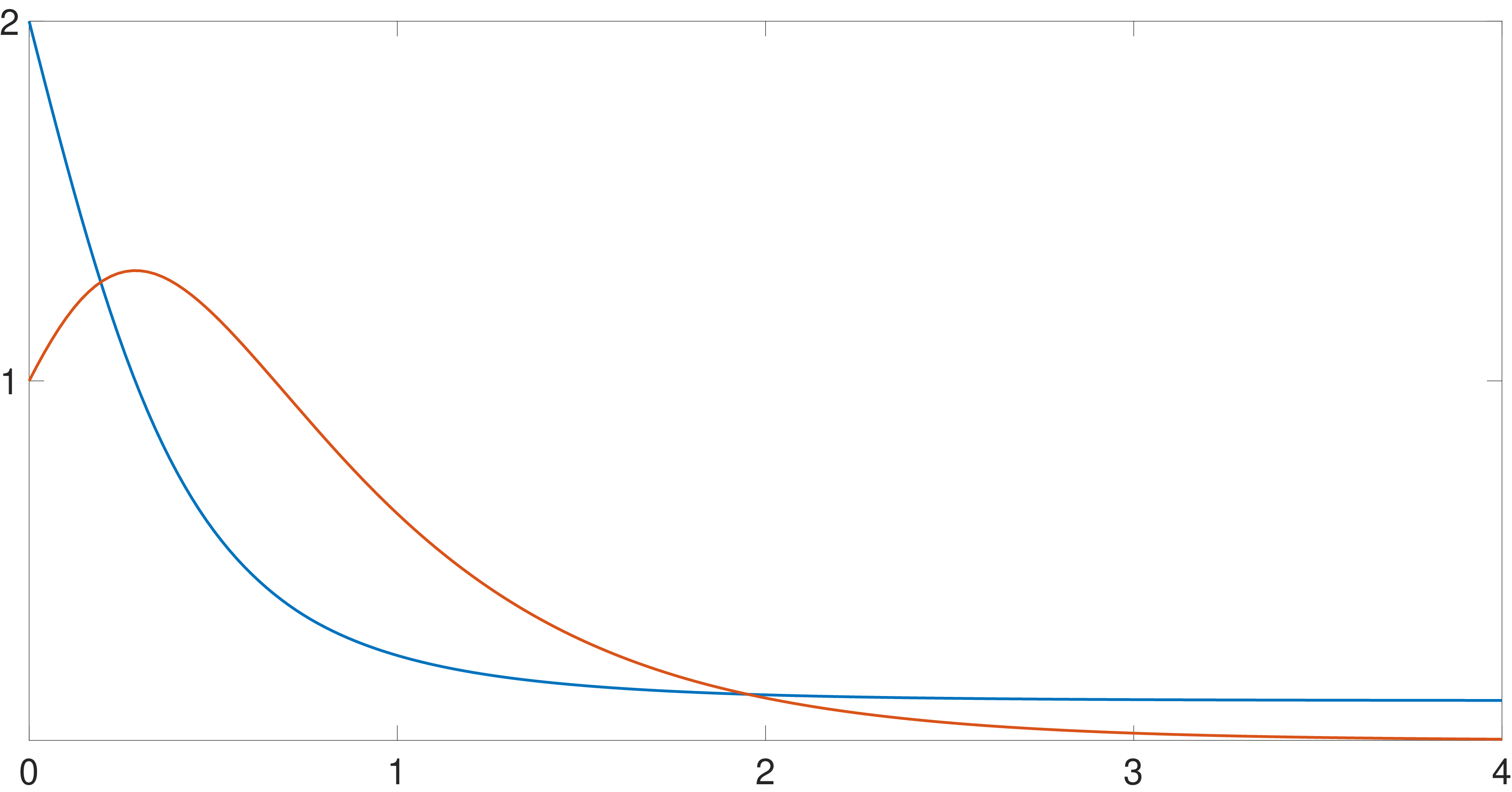}
 \caption{Solution of the SIR system with $S(0)=2$, $I(0)=1$, $\beta=2$, and $\gamma=2$. The graph of $S$ is shown in blue, and the graph of $I$ is shown in red. Note that $\beta S(0)> \gamma$ so that $I$ increases for an interval and then decreases to $0$.}\label{Iincthendec}
\end{figure}

\subsection{Controlled SIR}
\par In this note, we will consider the following analog of the SIR system
\be\label{SIRcontrol}
\begin{cases}
\dot S=-\beta S I-r S\\
\dot I=\beta SI -\gamma I 
\end{cases}
\ee 
where $r: [0,\infty)\rightarrow [0,1]$ represents a {\it vaccination rate control} of the SIR model.  This rate is conveniently limited by the upper bound 1; other constant upper bounds would lead to virtually the same theory which we present below.  Even though we have piecewise continuous controls $r$ in mind, it will be advantageous for us to consider \eqref{SIRcontrol} for each $r$ 
belonging to the collection  
$$
{\cal A}:=\{r\in L^\infty([0,\infty)): 0\le r(t)\le 1, \;\text{a.e. } t\ge 0\}
$$
of {\it admissible} vaccination rate controls. 

\par We'll see that for any $r\in {\cal A}$, there is a unique solution $S^r,I^r$ of \eqref{SIRcontrol} for given initial conditions $S^r(0), I^r(0)\ge 0$.  
Moreover, these solutions have the same qualitative properties of solutions to the uncontrolled SIR system \eqref{SIR} which we described above. 
In particular, the first time that the number of infectious individuals $I^r(t)$ falls below a given threshold $\mu>0$
$$
u^r=\inf\{t> 0: I^r(t)\le\mu\}
$$
is finite.  When $\mu$ is small, we can think of this time as an {\it eradication time}. In this paper, we will address the question: 
$$
\text{How do we choose a vaccination rate $r\in {\cal A}$ to minimize the eradication time $u^r$?}
$$

\par For this problem, Pontryagin's maximum principle \cite{MR0166037} asserts the following necessary conditions on an optimal vaccination rate $r$.  
\\\\
{\bf Necessary conditions for an optimal vaccination rate $r\in {\cal A}$}. There are absolutely continuous $P,Q:[0,u^r]\rightarrow \R$
such that the following statements hold.  
\begin{enumerate}[$(i)$]

\item $P, Q$ satisfy the ODE 
\be
\begin{cases}
\dot P(t) = (\beta I^r(t)+r(t))P(t)-\beta I^r(t)Q(t)\\
\dot Q(t) = \beta S^r(t)P(t)+(\gamma-\beta S^r(t))Q(t)
\end{cases}
\ee
for almost every $t\in [0,u^r]$.

\item $P(u^r)=0$ and $Q(u^r)\neq 0$.

\item $r(t)P(t)=P(t)^+$ for almost every $t\in [0,u^r]$.

\item For all $t\in [0,u^r]$, 
\be 
\beta S^r(t)I^r(t) P(t)+S^r(t)P(t)^++(\gamma-\beta S^r(t))I^r(t)Q(t)=1.
\ee
\end{enumerate}
\begin{rem}
When $r\in {\cal A}$ is an optimal vaccination rate,  we will specifically refer to the conditions above as the associated {\it necessary conditions $(i)-(iv)$}. 
\end{rem}

\par  In a recent paper \cite{MR3688684}, Bolzoni, Bonacini, Soresina, and Groppi used these necessary conditions to show that any optimal vaccination rate $r$ is of the form 
\be\label{SwitchingTimeControl}
r_\tau(t)=
\begin{cases}
0, \quad t\in [0,\tau]\\
1,\quad t\in (\tau,\infty)
\end{cases}
\ee
for some $\tau\ge 0$.  That is, any optimal vaccination rate will switch from not controlling the SIR system at all on $[0,\tau]$  to maximally controlling the SIR system on $(\tau,\infty)$. 
As a result, $\tau$ is interpreted as an optimal switching time. The corresponding vaccination rate $r_\tau$ is a ``bang-bang" control as it only takes on the extreme values in the 
interval $[0,1]$ in which each vaccination rate $r$ may assume.

\subsection{The dynamic programming approach}
\par In what follows, we will study this eradication time problem from the standpoint of dynamic programming.  To this end,  we will consider the eradication time function 
\be
u^r(x,y):=\inf\{t> 0: I^r(t)=\mu\}
\ee
for a given vaccination rate $r\in {\cal A}$. Here $S^r$ and $I^r$ satisfy \eqref{SIRcontrol} for this $r$ and initial conditions $S^r(0)=x\ge 0$ and $I^r(0)=y\ge \mu.$
A crucial property of $u^r$ is that for each $t\in [0,u^r(x,y)]$, 
\be\label{babyDPP}
u^r(x,y)=t+u^r(S^r(t),I^r(t)).
\ee
That is, after $t$ units of time, the time remaining for $I^r$ to decease to $\mu$ is simply $u^r(x,y)-t$. 

\par The corresponding {\it value function} is defined as
\be
u(x,y):=\min_{r\in {\cal A}} u^r(x,y)
\ee
for $x\ge 0$ and $y\ge \mu$.  Employing \eqref{babyDPP}, we will show that $u$ satisfies the {\it dynamic programming principle} 
\be
u(x,y)=\min_{r\in {\cal A}}\{t+u(S^r(t),I^r(t))\}
\ee
for $t\ge 0$. A direct consequence of dynamic programming is that $u$ is a viscosity solution of a Hamilton-Jacobi-Bellman (HJB) equation
\be\label{HJB}
\beta xy \partial_xu+x(\partial_xu)^++(\gamma-\beta x)y\partial_yu=1
\ee
in $(0,\infty)\times(\mu,\infty)$.

\par It follows from the definition of $u$ that 
\be\label{BCsegment}
u(x,\mu)=0\; \text{ for $0\le x\le \gamma/\beta$}
\ee
and
\be\label{BCsegment2}
u(0,y)=\frac{1}{\gamma}\ln\left(\frac{y}{\mu}\right)\; \text{ for $y\ge \mu$}.
\ee
Moreover, we will show
\be\label{BCatInfinity}
\lim_{ x+y\rightarrow\infty}u(x,y)=\infty.
\ee
It turns out that $u$ is the unique continuous viscosity solution of \eqref{HJB} which satisfies these three conditions.
\begin{thm}\label{thm0}
The value function $u$ is the unique continuous viscosity solution of \eqref{HJB} which satisfies \eqref{BCsegment}, \eqref{BCsegment2}, and 
\eqref{BCatInfinity}.
\end{thm}

\par We will also argue that $u$ is twice differentiable almost everywhere and its Hessian is essentially bounded from above in each compact subset of $(0,\infty)\times(\mu,\infty)$.  This follows from the following theorem.

\begin{thm}\label{thm1}
For each convex, compact $K\subset (0,\infty)\times(\mu,\infty)$, 
there is a constant $L$ such that
$$
u(x,y)-\frac{L}{2}(x^2+y^2)
$$ 
is concave on $K$.
\end{thm}

\par Using the fact that each optimal control is of the form $r_\tau$ for some $\tau$, we will also derive the following representation of the value function. Note that this allows 
us to give a sufficient condition for a vaccination rate $r_\tau$ to be optimal. 
\begin{thm}\label{thm2}
Suppose $S, I$ is the solution of the SIR system \eqref{SIR} with $S(0)=x\ge 0$ and $I(0)=y\ge \mu$. Then 
\be\label{SecondFormulaU}
u(x,y)=\min_{\tau\ge 0}\{\tau+u^{r_0}(S(\tau),I(\tau))\}.
\ee
Moreover, any $\tau$ for which the minimum in \eqref{SecondFormulaU} is achieved corresponds to an optimal vaccination rate $r_\tau$, and
\be\label{minTimeTau}
\tau^*=\min\{t\ge 0: u(S(t),I(t))=u^{r_0}(S(t),I(t)) \}
\ee
is a minimizing time. 
\end{thm}
\begin{rem}
In \eqref{SecondFormulaU}, $r_0$ is the switching time \eqref{SwitchingTimeControl} with $\tau=0$. 
\end{rem}

\par Equation \eqref{SecondFormulaU} also implies $u$ is a viscosity solution of 
\be\label{HJB2}
\max\{\beta xy \partial_xu+(\gamma-\beta x)y\partial_yu-1, u-u^{r_0}\}=0
\ee
in $(0,\infty)\times(\mu,\infty)$. Finding a solution of this PDE is sometimes called a ``free boundary" problem as if we happened to know the region 
$$
{\cal S}:=\{(x,y)\in(0,\infty)\times(\mu,\infty): u(x,y)=u^{r_0}(x,y) \},
$$
we could solve the PDE 
$$
\beta xy \partial_xu+(\gamma-\beta x)y\partial_yu=1
$$
in the complement of ${\cal S}$ subject to the boundary condition $u=u^{r_0}$ in order to obtain $u$. In addition, we can use this set to express $\tau^*$ defined in \eqref{minTimeTau} as the first time $t$ for which $(S(t),I(t))\in {\cal S}$. 

\par Finally, we will employ the value function $u$ to verify the necessary conditions which follow from Pontryagin's maximum principle. 
\begin{thm}\label{thm3}
Let $x> 0$ and $y>\mu$ and choose $r\in {\cal A}$ such that 
$$
u:=u(x,y)=u^r(x,y).
$$
Define
$$
P(t)=\partial_xu^r(S^r(t),I^r(t))\quad\text{and}\quad Q(t)=\partial_yu^r(S^r(t),I^r(t))
$$
for $t\in [0,u]$. Then $P,Q$ satisfy the necessary conditions $(i)-(vi)$. 
\end{thm}

\par  As hinted at above, the paper by Bolzoni, Bonacini, Soresina, and Groppi \cite{MR3688684} was a major inspiration of this work.   However, we would also like to emphasize that we gained perspective and learned techniques for time-optimal control by studying the notes of Evans \cite{EvansNotes} and the monographs by Bardi and Capuzzo-Dolcetta \cite{MR1484411}, Fleming and Soner \cite{MR2179357}, Fleming and Rishel \cite{MR0454768}, and Cesari \cite{MR688142}.  We would also like to point out that there have been several recent papers  \cite{MR3328003,MR3688684,MR3985528,MR2771181, MR3629459,MR1822058,HJBpaper} which consider control problems involving compartmental models.  We hope that our work adds in a positive way to this trend. 

\par This paper is organized as follows. In section \ref{ExistenceSection}, we will study the controlled SIR system \eqref{SIRcontrol} and verify the existence of an optimal vaccination rate for 
any given initial conditions.  Then in sections \ref{HJBsect} and \ref{UniqueSect}, we will show $u$ is the unique viscosity solution of the HJB equation \eqref{HJB} which satisfies conditions \eqref{BCsegment}, \eqref{BCsegment2}, and 
\eqref{BCatInfinity}.
 Next, we'll   study the differentiability of $u$ and prove Theorem \ref{thm1} in section \ref{SemiConcaveSect}. In section \ref{OptSwtichSect}, we derive formula \eqref{SecondFormulaU} and consider the PDE \eqref{HJB2}. Finally in section \ref{PMPsect}, we will prove Theorem \ref{thm3}.

\section{Existence of an optimal control}\label{ExistenceSection}
In this preliminary section, we will explain that there always is a solution of the controlled SIR system \eqref{SIRcontrol} and derive several properties of solutions. In particular, we will show solutions depend continuously on their initial conditions and on the control.  We will use this continuity to show that an optimal vaccination rate exists for our eradication time problem.

\begin{lem}\label{ExistenceLemmaSIRcontrol}
For any $x,y\ge 0$ and $r\in {\cal A}$, there is a unique solution
$$
S,I: [0,\infty)\rightarrow \R
$$
of the controlled SIR equations \eqref{SIRcontrol} with $S(0)=x$ and $I(0)=y$.  Moreover, $S, I$, and $\dot I$ are Lipschitz continuous. 
\end{lem}
\begin{proof}
By Caratheodory's Theorem (Theorem 5.1 in section I.5 of \cite{MR587488}),
there is an absolutely continuous local solution $S, I: [0,T)\rightarrow \R$. We also set 
$$
R(t):=\displaystyle\gamma \int^t_0I(\tau)d\tau, \quad t\in [0,T),
$$
so 
$$
S(t)+I(t)+R(t)= S(0)+I(0)+R(0)=x+y
$$
for $t\in [0,T)$. In view of \eqref{SIRcontrol}, 
\be
S(t)= xe^{\displaystyle- \int^t_0\beta I(\tau)+r(\tau)d\tau}\;\;\text{and}\;\; I(t)=ye^{\displaystyle \int^t_0(\beta S(\tau)-\gamma)d\tau}.
\ee
Thus, $S(t),I(t)\ge 0$ for $t\in [0,T)$. It follows that 
$$
0\le S(t), I(t)\le x+y, \quad t\in [0,T).
$$
It is then possible to continue this solution to all of $[0,\infty)$ (Theorem 5.2 in section I.5 of \cite{MR587488}).  
Given that $S(t), I(t)$ are bounded, it is also not hard to check that this solution is unique. 

\par Note that 
$$
0\ge \dot S(t)\ge -\beta (x+y)^2-(x+y)
$$
for almost every $t\ge 0$.  Thus, $S$ is Lipschitz continuous. We also note 
$$
|\dot I(t)|\le \beta (x+y)^2+\gamma(x+y)
$$
for all $t\ge 0$, so $I$ is Lipschitz continuous. Differentiating the second equation in \eqref{SIRcontrol} we see that 
$$
\ddot I(t)=-\beta (\beta I(t)+r(t))S(t)I(t)+(\beta S(t)-\gamma)^2I(t)
$$
for almost every $t\ge 0$. Thus, 
\be
|\ddot I(t)|\le \beta (\beta (x+y)+1)(x+y)^2+2(\beta^2(x+y)^2+\gamma^2)(x+y)
\ee
for almost every $t\ge 0$. It follows that $\dot I$ is also Lipschitz continuous.  
\end{proof}
\begin{lem}\label{LimitLemma}
Suppose $S,I$ is a solution of \eqref{SIRcontrol} with $S(0)\ge 0$ and $I(0)>0$ for some $r\in {\cal A}$. Then the limit
\be
\lim_{t\rightarrow\infty}S(t)\in\left[0,\frac{\gamma}{\beta}\right)
\ee
exists and
$$
\lim_{t\rightarrow\infty}I(t)=0.
$$ 
\end{lem}
\begin{proof}
From the proof of the previous lemma, we have 
$$
\gamma\int^\infty_0 I(\tau)d\tau\le S(0)+I(0).
$$
It follows that there is a sequence of times $t_k\nearrow\infty$ such that $\lim_{k\rightarrow\infty}I(t_k)=0$.  Also note that 
\be\label{IformulaGoesZero}
I(t)=I(0)+\beta \int^t_0 S(\tau)I(\tau)d\tau -\gamma\int^t_0 I(\tau)d\tau
\ee
for all $t\ge 0$. Choosing $t=t_k\rightarrow\infty$ and sending $k\rightarrow\infty$ gives  
$$
0=I(0)+\beta \int^\infty_0 S(\tau)I(\tau)d\tau -\gamma\int^\infty_0 I(\tau)d\tau.
$$
Now we can send $t\rightarrow\infty$ in \eqref{IformulaGoesZero} to get 
$$
\lim_{t\rightarrow\infty}I(t)=I(0)+\beta \int^\infty_0 S(\tau)I(\tau)d\tau -\gamma\int^\infty_0 I(\tau)d\tau=0.
$$ 

\par Suppose $S(0)>0$ or else $S(t)=0$ for all $t\ge 0$. For $S(0)>0$, $S$ is decreasing and positive, thus $\lim_{t\rightarrow\infty}S(t)$ exists.  It follows that if $\beta S(0)\le \gamma$, then 
\be\label{Slimitlessthan}
\beta\lim_{t\rightarrow\infty}S(t)<\gamma.
\ee
Otherwise, $I$ is initially increasing and must have a critical point at a time $t_0>0$ with $\beta S(t_0)=\gamma.$
As $S$ is decreasing, \eqref{Slimitlessthan} holds in this case, as well.
\end{proof}

\par We emphasize that since $S$ is decreasing, $I$ can have at most one critical point.  We'll also record one more fact which essentially follows from the 
proof above. 
\begin{cor}\label{Sulessthan}
Let $r\in {\cal A}$ and suppose $S, I$ is the corresponding solution of \eqref{SIRcontrol}
which satisfies $S(0)=x> 0$ and $I(0)=y>\mu$. Then 
$$
\beta S(u)<\gamma
$$
where $u=\inf\{t> 0: I(t)= \mu\}.$
\end{cor}
\begin{proof}
As $S$ is decreasing and $u>0$, $\beta x\le \gamma$ implies 
$$
\beta S(u)-\gamma<\beta x-\gamma\le 0.
$$
If $\beta x>\gamma$, $I$ will initially increase. Let $t>0$ be the maximum time for $I$. At this time  $\beta S(t)=\gamma$ and $I(t)> \mu.$
It follows that $t<u$ and 
$$
\beta S(u)-\gamma<\beta S(t)-\gamma=0.
$$
\end{proof}

\par We recall that a sequence $(r^k)_{k\in \N}\subset L^\infty([0,\infty))$ converges weak* to $r$ if 
$$
\lim_{k\rightarrow\infty}\int^\infty_0 g(t)r^k(t)dt=\int^\infty_0 g(t)r(t)dt
$$
for each $g\in L^1([0,\infty))$. Moreover, any sequence $(r^k)_{k\in \N}\subset L^\infty([0,\infty))$ with 
$$
\sup_{k\in \N}\|r^k\|_{L^\infty([0,\infty))}<\infty
$$
has a subsequence which converges weak*.  In particular, the control set ${\cal A}$ is weak* compact. We can use this notion to show solutions of \eqref{SIRcontrol} depend continuously on $r$ and their initial conditions. 

\begin{prop}\label{CompactnessProp}
Suppose $x^k$, $y^k\ge 0$, $r^k\in {\cal A}$ for each $k\in \N$, and 
\be
\begin{cases}
x^k\rightarrow x^\infty\\
y^k\rightarrow y^\infty\\
r^k\rightarrow r^\infty\;\text{weak*}
\end{cases}
\ee
as $k\rightarrow\infty$. If $S^k,I^k$ is the solution of \eqref{SIRcontrol} with $r=r^k$, $S^k(0)=x^k$, and  $I^k(0)=y^k$,  
then 
$$
\begin{cases}
S^k(t)\rightarrow S^\infty(t)\\
I^k(t)\rightarrow I^\infty(t)\\
\end{cases}
$$
uniformly for $t$ belonging to bounded subintervals of $[0,\infty)$.  Here $S^\infty,I^\infty$ is the solution of \eqref{SIRcontrol} with $r=r^\infty$, $S^\infty(0)=x^\infty$, and  $I^\infty(0)=y^\infty$.
\end{prop} 
\begin{proof}
We showed in Lemma \ref{ExistenceLemmaSIRcontrol} that 
$$
0\le S^k(t), I^k(t)\le x^k+y^k, \quad t\ge 0
$$
and 
$$
\begin{cases}
|\dot S^k(t)|\le \beta(x^k+y^k)^2+(x^k+y^k)\\\\
|\dot I^k(t)|\le \beta(x^k+y^k)^2+\gamma(x^k+y^k)
\end{cases}
$$
for almost every $t\ge 0$.  As $x^k$ and $y^k$ are convergent, the sequences $(S^k)_{k\in\N}$ and 
$(I^k)_{k\in\N}$ of continuous functions are both uniformly bounded and uniformly equicontinuous. The Arzel\`a-Ascoli 
Theorem implies that there are locally uniformly convergent sequences $(S^{k_j})_{j\in\N}$ and $(I^{k_j})_{j\in\N}$.  
Let us write $S, I: [0,\infty)\rightarrow\infty$ for their respective limits.   

\par Clearly $S(0)=x$ and $I(0)=y$. By Lemma \ref{ExistenceLemmaSIRcontrol}, it suffices to show $S,I$ satisfy  \eqref{SIRcontrol} with $r=r^\infty$.
To this end, we note that $S^k$ and $I^k$ satisfy 
$$
S^k(t)=x^k-\beta \int^t_0 S^k(\tau)I^k(\tau)d\tau -\int^t_0 r^k(\tau)S^k(\tau)d\tau
$$
and 
$$
I^k(t)=y^k+\beta \int^t_0 S^k(\tau)I^k(\tau)d\tau -\gamma\int^t_0 I^k(\tau)d\tau
$$
for each $t\ge 0$. Employing the weak* convergence of $r^{k}$ and the the local uniform convergence of $(S^{k_j})_{j\in\N}$ and $(I^{k_j})_{j\in\N}$, 
we can let $k=k_j\rightarrow\infty$ in the two identities above to conclude 
$$
S(t)=x-\beta \int^t_0 S(\tau)I(\tau)d\tau -\int^t_0 r^\infty(\tau)S(\tau)d\tau
$$
and 
$$
I(t)=y+\beta \int^t_0 S(\tau)I(\tau)d\tau -\gamma\int^t_0 I(\tau)d\tau
$$
for each $t\ge 0$. That is, $S=S^\infty$ and $I=I^\infty$. 
\end{proof}

\par Let us fix a threshold
$$
\mu>0
$$
and a pair of initial conditions
$$
x\ge 0\;\text{and}\; y\ge\mu.
$$
For a given $r\in {\cal A}$,  we will denote $S^r, I^r$ denote the solution of the  \eqref{SIRcontrol}
which satisfies $S^r(0)=x$ and $I^r(0)=y$. We define 
$$
u^r:=\inf\{t> 0: I^{r}(t)= \mu\}
$$
and now argue that a minimizing vaccination rate $r\in {\cal A}$ exists.

\begin{thm}\label{ExistenceTheorem}
There is $r^*\in {\cal A}$ such that 
\be\label{rstarIneq}
u^{r^*}\le u^r
\ee
for all $r\in {\cal A}$.
\end{thm}
\begin{proof}
Choose a minimizing sequence $(r^k)_{k\in \N}\subset {\cal A}$
$$
\inf_{r\in {\cal A}} u^r=\lim_{k\rightarrow\infty}u^{r^k}.
$$
Without any loss of generality we may assume that $r^k\rightarrow r^\infty$ weak* to some $r^\infty$ as this occurs 
for a subsequence. Let $S^k, I^k$ denote the solution of  \eqref{SIRcontrol} with $r=r^k$, $S^k(0)=x$, and $I^k(0)=y$. 
By Proposition \ref{CompactnessProp}, $S^k$ and $I^k$ converge locally uniformly to $S^\infty$ and $I^\infty$, respectively, the 
 solution of  \eqref{SIRcontrol} with $r=r^\infty$, $S^\infty(0)=x$, and $I^\infty(0)=y$.   
 
 \par Therefore, we can send $k\rightarrow\infty$ in $I^k(u^{r^k})=\mu$ to get
 $$
 I^\infty\left(\inf_{r\in {\cal A}} u^r \right)=\mu.
 $$
 That is, 
 $$
u^{r^\infty}\le  \inf_{r\in {\cal A}}u^r.
 $$
\end{proof}

\par We'll call any $r^*\in {\cal A}$ satisfying \eqref{rstarIneq} an {\it optimal vaccination rate} (for the SIR eradication time problem) with
initial conditions $S(0)=x\ge0$ and $I(0)=y\ge\mu$.  In the sections that follow, we will develop methods to characterize such rates.

\section{The HJB equation}\label{HJBsect}
We will now consider our time optimal control problem for varying initial conditions.  To this end, we will employ the value function
\be\label{ValueFunction}
u(x,y)=\min_{r\in {\cal A}} u^r(x,y)
\ee
discussed in the introduction. Here $u^r(x,y)=\inf\{t> 0: I^r(t)=\mu\}$ is the eradication time corresponding to a given vaccination rate $r\in {\cal A}$, and $S^r$ and $I^r$ satisfy \eqref{SIRcontrol} with $S^r(0)=x\ge 0$ and $I^r(0)=y\ge \mu$.  In this section, we will show that $u$ is a continuous viscosity solution of the HJB equation \eqref{HJB}.

\par Our first task will be to establish that $u$ is continuous on $[0,\infty)\times[\mu,\infty)$. To this end, we'll start by showing that $u^r$ is locally bounded uniformly in $r\in {\cal A}$.
\begin{lem} 
Let $r\in {\cal A}$. Then
\be\label{uarrUpperBound}
0\le u^r(x,y)\le \frac{x+y}{\mu\gamma}
\ee
for $x\ge 0$ and $y\ge \mu$. 
\end{lem}
\begin{proof}
Set
$$
w(x,y)=\frac{x+y}{\mu\gamma}
$$
and note $\beta xy \partial_xw +(\gamma-\beta x)y\partial_yw=y/\mu.$  As a result, 
\begin{align}
&\frac{d}{dt}w(S^r(t),I^r(t))\\
&=
\partial_xw(S^{r}(s),I^{r}(s))(-\beta S^{r}(t) I^{r}(t)-r(t)S^{r}(t))+\partial_yw(S^{r}(t),I^{r}(t))(\beta S^{r}(t)I^{r}(t) -\gamma I^{r}(t))\\
&=-\frac{1}{\mu\gamma}r(t)S^{r}(t)-\frac{1}{\mu}I^r(t)\\
&\le -\frac{1}{\mu}I^r(t)\\
&\le -1
\end{align}
for $t\in [0,u^r(x,y)]$.  Integrating from $0$ to $t=u^r(x,y)$ gives 
$$
w(S^r(u^r(x,y)),I^r(u^r(x,y)))-w(x,y) \le -u^r(x,y).
$$
And as $w$ is nonnegative, $w(x,y)\ge u^r(x,y)$. 
\end{proof}
\begin{cor}\label{ConvUany}
Suppose $x^k\ge 0$, $y^k\ge \mu$ and $r^k\in {\cal A}$ for each $k\in \N$, and 
\be
\begin{cases}
x^k\rightarrow x\\
y^k\rightarrow y\\
r^k\rightarrow r\;\text{weak*}
\end{cases}
\ee
as $k\rightarrow\infty$. Then
\be\label{Convergenceurkayxkayykay}
\lim_{k\rightarrow\infty}u^{r^k}(x^k,y^k)=u^{r}(x,y).
\ee
\end{cor}

\begin{proof}
Suppose $S^{r^k}, I^{r^k}$ is the solution of \eqref{SIRcontrol} with $S^{r^k}(0)=x^k$ and $I^{r^k}(0)=y^k.$   By Proposition \ref{CompactnessProp}, $S^{r^{k}}, I^{r^{k}}$ converge locally uniformly to $S^{r}, I^{r}$ as $k\rightarrow\infty$.  In view of the previous lemma, $u^{r^{k}}(x^{r^{k}},y^{r^{k}})$ is a bounded sequence; so there is a 
convergent subsequence $u^{r^{k_j}}(x^{r^{k_j}},y^{r^{k_j}})$ for which 
$$
t:=\liminf_{k\rightarrow\infty}u^{r^{k}}(x^{r^{k}},y^{r^{k}})=\lim_{j\rightarrow\infty}u^{r^{k_j}}(x^{r^{k_j}},y^{r^{k_j}}).
$$
As $I^{r^{k}}(u^{r^{k}}(x^{r^{k}},y^{r^{k}}))=\mu$ for each $k\in \N$,
$$
I^{r}(t)=\lim_{j\rightarrow\infty}I^{r^{k_j}}(u^{r^{k_j}}(x^{r^{k_j}},y^{r^{k_j}}))=\mu.
$$
By the definition of $u^r(x,y)$,
$$
u^r(x,y)\le t=\liminf_{k\rightarrow\infty}u^{r^{k}}(x^{r^{k}},y^{r^{k}}).
$$

\par We can also select a convergent subsequence $u^{r^{k_\ell}}(x^{r^{k_\ell}},y^{r^{k_\ell}})$ such that
$$
s:=\limsup_{k\rightarrow\infty}u^{r^{k}}(x^{r^{k}},y^{r^{k}})=\lim_{\ell\rightarrow\infty}u^{r^{k_\ell}}(x^{r^{k_\ell}},y^{r^{k_\ell}}).
$$
As above, we find $I^{r}(s)=\mu$.  If $y>\mu$ or if $0\le x\le \gamma/\beta$ and $y=\mu$, then the only solution of $I^r(\tau)=\mu$ is $\tau=u^r(x,y)$.  In particular, $s=u^r(x,y)$.  Otherwise, if $x>\gamma/\beta$ and $y=\mu$,  $I^r(\tau)=0$ has two solutions $\tau=0$ and $\tau=u^r(x,y)$. Thus, $s\le u^r(x,y)$ with either possibility.  It follows that 
$$
s=\limsup_{k\rightarrow\infty}u^{r^{k}}(x^{r^{k}},y^{r^{k}})\le u^r(x,y).
$$
We conclude \eqref{Convergenceurkayxkayykay}.
\end{proof}
\par Note the limit \eqref{Convergenceurkayxkayykay} implies $u^r$ is continuous on $[0,\infty)\times [\mu,\infty)$ for each $r\in {\cal A}$. The value function $u$ also inherits this continuity.
\begin{prop}
The value function $u$ is continuous at $(x,y)\in [0,\infty)\times [\mu,\infty)$.
\end{prop}
\begin{proof}
Suppose $x^k\ge 0$, $y^k\ge \mu$, and $x^k\rightarrow x, y^k\rightarrow y$ as $k\rightarrow\infty$, and select $r^k\in {\cal A}$ for which 
\be\label{Firstukayidentity}
u(x^k,y^k)=u^{r^k}(x^k,y^k)
\ee
($k\in \N$). We may select an increasing sequence of positive integers $k=k_j\rightarrow\infty$ such that 
 $$
 \liminf_{k\rightarrow\infty}u(x^k,y^k)=\lim_{j\rightarrow\infty}u^{r^{k_j}}(x^{k_j},y^{k_j})
 $$
and for which $r^{k_j}$ converges weak$^*$ to some $r^*\in {\cal A}$.  Using \eqref{Convergenceurkayxkayykay} gives 
$$
\liminf_{k\rightarrow\infty}u(x^k,y^k)=\lim_{j\rightarrow\infty}u^{r^{k_j}}(x^{k_j},y^{k_j})=u^{r^*}(x,y)\ge u(x,y).
$$
\par By \eqref{Firstukayidentity}, we also have $u(x^k,y^k)\le u^{r}(x^k,y^k)$ for all $k\in\N$ and each $r\in {\cal A}$. Using  \eqref{Convergenceurkayxkayykay} again gives 
$$
\limsup_{k\rightarrow\infty}u(x^k,y^k)\le u^r(x,y).
$$
Since $r\in {\cal A}$ is arbitrary, $\limsup_{k\rightarrow\infty}u(x^k,y^k)\le u(x,y)$. That is, 
$$
\limsup_{k\rightarrow\infty}u(x^k,y^k)= u(x,y)= \liminf_{k\rightarrow\infty}u(x^k,y^k).
$$
It follows that $u$ is continuous at $(x,y)$.
\end{proof}

\par Next, we will establish dynamic programming and then use this property to verify that $u$ is a viscosity solution of the 
HJB equation \eqref{HJB}.  We note that these types of results have been considered more generally elsewhere. An excellent reference for dynamic programming in time optimal control is Chapter IV of the monograph by Bardi and Capuzzo-Dolcetta \cite{MR1484411}.   In addition to \cite{MR1484411}, another standard reference for viscosity solutions is the monograph by Fleming and Soner  \cite{MR2179357}.

\begin{prop}\label{DPPprop}
Let $x\ge 0$ and $y\ge \mu$. If $t\in [0,u(x,y)]$,
\be\label{DPP}
u(x,y)=\min_{r\in {\cal A}}\{t+u(S^r(t),I^r(t))\}.
\ee
The minimum is attained by any $r\in {\cal A}$ such that $u(x,y)=u^{r}(x,y)$; and for any such $r$, 
\be
u(S^{r}(t),I^{r}(t))=u^{r}(S^{r}(t),I^{r}(t)).
\ee
\end{prop}

\begin{proof}
We note that $0\le t\le u^r(x,y)$ for $r\in {\cal A}$ and 
$$
u^r(x,y)=t+u^r(S^r(t), I^r(t)).
$$
For any $r^*\in {\cal A}$ such that $u(x,y)=u^{r^*}(x,y)$, 
\be\label{DPPreallymin}
u(x,y)=t+u^{r^*}(S^{r^*}(t),I^{r^*}(t))\ge t+u(S^{r^*}(t),I^{r^*}(t))\ge \inf_{r\in {\cal A}}\{t+u(S^r(t),I^r(t))\}.
\ee

\par To derive the opposite inequality, we fix $r\in {\cal A}$ and choose $r^*\in {\cal A}$ such that 
\be\label{SettingUpkeyTimeIdentity}
u(S^r(t),I^r(t))=u^{r^*}(S^r(t),I^r(t)).
\ee
Let us also define 
\be\label{overlinearrr}
\overline{r}(s)=
\begin{cases}
r(s), \quad &0\le s\le t\\
r^*(s-t), \quad &t\le s<\infty.
\end{cases}
\ee
We claim that 
\be\label{keyTimeIdentity}
u^{\overline{r}}(S^{\overline{r}}(t), I^{\overline{r}}(t))=u^{r^*}(S^r(t), I^r(t)).
\ee
In particular, this common number is the first time $s=s^*$ the solution of 
\be\label{SIRcontrolex}
\begin{cases}
\dot X(s)=-\beta X(s) Y(s)-r^*(s-t) X(s)\\
\dot Y(s)=\beta X(s)Y(s) -\gamma Y(s) 
\end{cases} \; (s>t)
\ee 
with $X(t)=S^r(t), Y(t)=I^r(t)\ge \mu$ satisfies $Y(s)=\mu$. That $u^{\overline{r}}(S^{\overline{r}}(t), I^{\overline{r}}(t))=s^*$ follows from \eqref{overlinearrr}; note in particular that $S^{\overline{r}}(s)=X(s)$ and $ I^{\overline{r}}(s)=Y(s)$ for $s\ge t$. 
The right hand side of \eqref{keyTimeIdentity} also equals $s^*$ once we note $S(\tau)=X(\tau+t)$ and $I(\tau)=Y(\tau+t)$ solve \eqref{SIRcontrol} 
with $r^*$ and satisfy $S(0)=S^r(t)$ and $I(0)=I^r(t)$.

\par By \eqref{SettingUpkeyTimeIdentity} and \eqref{keyTimeIdentity},
\begin{align}
u(x,y)&\le u^{\overline{r}}(x,y)\\
&= t+u^{\overline{r}}(S^{\overline{r}}(t), I^{\overline{r}}(t))\\
&= t+u^{r^*}(S^r(t),I^r(t))\\
&=t+u(S^r(t),I^r(t)).
\end{align}
That is, 
$$
u(x,y)\le  \inf_{r\in {\cal A}}\{t+u(S^r(t),I^r(t))\}.
$$
In view of \eqref{DPPreallymin}, equality holds in this inequality; the infimum is achieved for any $r^*\in {\cal A}$ such that $u(x,y)=u^{r^*}(x,y)$, 
and  we also note $u^{r^*}(S^{r^*}(t),I^{r^*}(t))=u(S^{r^*}(t),I^{r^*}(t))$.
\end{proof}

\par A corollary of dynamic programming is that the value function $u$ is a viscosity solution of \eqref{HJB} 
\be
\beta xy \partial_xu+x(\partial_xu)^++(\gamma-\beta x)y\partial_yu=1
\ee
in $(0,\infty)\times(\mu,\infty)$.
\begin{cor}\label{ViscProperty}
The value function $u$ is a viscosity solution of the HJB equation \eqref{HJB}.
\end{cor}
\begin{proof}
Fix $x_0>0$ and $y_0>\mu$ and suppose that $u-\varphi$ has a local maximum at $(x_0,y_0)$; here $\varphi$ is a continuously differentiable function defined on a neighborhood of $(x_0,y_0)$.  We further assume $a\in [0,1]$, $r(t)=a$ for all $t\ge 0$, and  $S^r$ and $I^r$ is the solution of \eqref{SIRcontrol} with $S^r(0)=x_0$ and $I^r(0)=y_0$.  It follows that 
$$
(u-\varphi)(S^r(t),I^r(t))\le (u-\varphi)(x_0,y_0)
$$
for all $t\ge 0$ small.  By dynamic programming $u(x_0,y_0)\le t+ u(S^r(t),I^r(t))$ for $t\ge 0$. As a result,
$$
-t\le u(S^r(t),I^r(t))-u(x_0,y_0)\le \varphi(S^r(t),I^r(t))-\varphi(x_0,y_0)
$$
for all $t\ge 0$ small. In particular, 
\begin{align}
-1&\le \left.\frac{d}{dt} \varphi(S^r(t),I^r(t))\right|_{t=0}\\
&=-(\beta x_0y_0 +ax_0) \partial_x\varphi(x_0,y_0)-(\gamma-\beta x_0)y_0\partial_y\varphi(x_0,y_0).
\end{align}
Rearranging this inequality gives 
$$
\beta x_0y_0 \partial_x\varphi(x_0,y_0)+x_0a\partial_x\varphi(x_0,y_0)+(\gamma-\beta x_0)y_0\partial_y\varphi(x_0,y_0)\le 1.
$$
And taking the supremum over all $a\in [0,1]$ we find 
$$
\beta x_0y_0 \partial_x\varphi(x_0,y_0)+x_0(\partial_x\varphi(x_0,y_0))^++(\gamma-\beta x_0)y_0\partial_y\varphi(x_0,y_0)\le 1.
$$

\par Conversely, suppose $u-\psi$ has a local minimum at $(x_0,y_0)$ and $r^*\in{\cal A}$ such that $u(x_0,y_0)=u^{r^*}(x_0,y_0)$.  Here 
$\psi$ is a continuously differentiable function defined on a neighborhood of $(x_0,y_0)$.  By Proposition \ref{DPPprop},
$$
u(x_0,y_0)= t+ u(S^{r^*}(t),I^{r^*}(t))
$$
for all small $t> 0$, where $S^{r^*}$ and $I^{r^*}$ is the solution of \eqref{SIRcontrol} with $S^{r^*}(0)=x_0$ and $I^{r^*}(0)=y_0$.  Consequently, 
$$
-t=u(S^{r^*}(t),I^{r^*}(t))-u(x_0,y_0)\ge \psi(S^{r^*}(t),I^{r^*}(t))-\psi(x_0,y_0)
$$
for all small $t> 0$.  As $0\le r^*\le 1$,
\begin{align}
-1&\ge \frac{1}{t}( \psi(S^{r^*}(t),I^{r^*}(s))-\psi(S^{r^*}(0),I^{r^*}(s)))\\
&= \frac{1}{t}\int^t_0\frac{d}{ds}\psi(S^{r^*}(s),I^{r^*}(s))ds\\ 
&=\frac{1}{t}\int^t_0\partial_x\psi(S^{r^*}(s),I^{r^*}(s))(-\beta S^{r^*}(s) I^{r^*}(s)-r^{*}(s)S^{r^*}(s))+ \\
&\hspace{2in}\partial_y\psi(S^{r^*}(s),I^{r^*}(s))(\beta S^{r^*}(s)I^{r^*}(s) -\gamma I^{r^*}(s)) ds\\
&\ge \frac{1}{t}\int^t_0\left[-\beta S^{r^*}(s) I^{r^*}(s)\partial_x\psi(S^{r^*}(s),I^{r^*}(s))- S^{r^*}(s)\partial_x\psi(S^{r^*}(s),I^{r^*}(s))^+ + \right.\\
&\hspace{1.5in}\left. (\beta S^{r^*}(s)I^{r^*}(s)-\gamma I^{r^*}(s))\partial_y\psi(S^{r^*}(s),I^{r^*}(s)) \right]ds.
\end{align}
Sending $t\rightarrow 0^+$ gives 
$$
-1\ge -\beta x_0y_0 \partial_x\psi(x_0,y_0)-x_0(\partial_x\psi(x_0,y_0))^+-(\gamma-\beta x_0)y_0\partial_y\psi(x_0,y_0).
$$
That is
$$
\beta x_0y_0 \partial_x\psi(x_0,y_0)+x_0(\partial_x\psi(x_0,y_0))^++(\gamma-\beta x_0)y_0\partial_y\psi(x_0,y_0)\ge 1.
$$
\end{proof}

\par We will now establish \eqref{BCatInfinity} which asserts 
\be
\lim_{ x+y\rightarrow\infty}u(x,y)=\infty.
\ee
This will be a direct consequence of the following lemma. 
\begin{lem}
For each $x\ge 0$ and $y\ge \mu$, 
\be\label{LB2u}
u(x,y)\ge \frac{\ln(x+y)-\ln(\gamma/\beta+\mu)}{\max\{\gamma,1\}}.
\ee
\end{lem}
\begin{proof}
Set 
$$
w(x,y)=\frac{\ln(x+y)-\ln(\gamma/\beta+\mu)}{\max\{\gamma,1\}}.
$$
Observe 
\be
\beta xy \partial_xw+x(\partial_xw)^++(\gamma-\beta x)y\partial_yw =\frac{1}{\max\{\gamma,1\}}\frac{x+\gamma y}{x+y}\le 1
\ee
and 
\be\label{BC2}
w(x,\mu)\le 0\;\text{ for $0\le x\le \gamma/\beta$}.
\ee
It then follows that if $r\in {\cal A}$,
\begin{align}
&\frac{d}{dt}w(S^r(t),I^r(t))\\
&=
\partial_xw(S^{r}(s),I^{r}(s))(-\beta S^{r}(t) I^{r}(t)-r(t)S^{r}(t))+\partial_yw(S^{r}(t),I^{r}(t))(\beta S^{r}(t)I^{r}(t) -\gamma I^{r}(t))\\
&\ge -\beta S^{r}(t) I^{r}(t)\partial_xw(S^{r}(s),I^{r}(s)) -S^{r}(t)\partial_xw(S^{r}(t),I^{r}(t))^+\\
&\hspace{2in}-(\gamma-\beta S^{r}(t)) I^{r}(t)\partial_yw(S^{r}(t),I^{r}(t))\\
&\ge-1
\end{align}
for almost every $t\ge 0$. Integrating from $t=0$ to $t=u^r(x,y)$ gives 
$$
w(S^r(u^r(x,y)),I^r(u^r(x,y)))-w(x,y) \ge -u^r(x,y).
$$
As $I^r(u^r(x,y))=\mu$, we can apply \eqref{BC2} to find $w(x,y)\le u^r(x,y)$. Since $r\in{\cal A}$ is arbitrary, $w(x,y)\le u(x,y).$
\end{proof}

\section{Uniqueness}\label{UniqueSect}
In this section, we will argue that the value function $u$ is the unique continuous viscosity solution of the HJB equation \eqref{HJB} in $(0,\infty)\times(\mu,\infty)$ which satisfies conditions \eqref{BCsegment}, \eqref{BCsegment2}, and \eqref{BCatInfinity}; recall that this is the statement of Theorem \ref{thm0}.  To this end, 
we will adapt the technique used to prove Theorem 2.6 in Chapter IV section 2 of \cite{MR1484411}, which is in turn based upon reference \cite{MR1001919}.  Theorem 2.6 in Chapter IV section 2 of \cite{MR1484411} is a general result on the comparison of viscosity sub- and supersolutions to HJB equations arising in time optimal control. In this  general setting, the domain of the time function is the collection of all points for which the associated control ODE has a solution which arrives at a given target in a finite time. The main idea is to change variables so that the corresponding HJB equation is proper. In our framework, this can be accomplished by 
setting
$$
v:=e^{-u}.
$$

\par In particular, we note that $v$ is a positive, continuous viscosity solution of 
\be\label{HJBvee}
v+\beta xy \partial_xv-x(\partial_xv)^-+(\gamma-\beta x)y\partial_yv=0
\ee
in $(0,\infty)\times(\mu,\infty)$.  In view of \eqref{BCsegment}, \eqref{BCsegment2}, and \eqref{BCatInfinity}, $v$ additionally satisfies 
\be\label{BCsegmentv}
v(x,\mu)=1
\ee
for $0\le x\le \gamma/\beta$,
\be\label{BCsegment2v}
v(0,y)=\left(\frac{y}{\mu}\right)^{-1/\gamma}
\ee
for $y\ge \mu$, and 
\be\label{BCatInfinityv}
\text{$v$ is bounded from above}.
\ee
As a result, in order to conclude Theorem \ref{thm0},  it suffices to prove the following claim. 

\begin{prop}\label{VeeProp}
Assume $v^1,v^2: [0,\infty)\times[\mu,\infty)\rightarrow(0,\infty)$ are continuous viscosity solutions of \eqref{HJBvee} in $(0,\infty)\times(\mu,\infty)$ which satisfy \eqref{BCsegmentv}, \eqref{BCsegment2v}, and \eqref{BCatInfinityv}. Then $v^1\equiv v^2$.
\end{prop}

\par We will verify uniqueness by employing the celebrated ``doubling the variables" argument of Crandall and Lions \cite{MR690039}. Before carrying out these details, we will show how to deduce uniqueness under the additional assumption that $v^1$ and $v^2$ are continuously differentiable in $(0,\infty)\times(\mu,\infty)$. This will motivate 
the subsequent viscosity solutions argument. 

\begin{proof}[Proof of Proposition \ref{VeeProp} assuming $v^1,v^2$ are continuously differentiable] Choose $g: [0,\infty)\rightarrow [0,\infty)$ to be any smooth, nondecreasing function which vanishes on $[0,\gamma/\beta]$ and is positive on $(\gamma/\beta,\infty)$. Next set 
$$
w(x,y):=
\begin{cases}
\displaystyle x+y+\frac{g(x)}{y-\mu},\quad & x\ge 0,\; y>\mu\\
x+y,\quad & 0\le x\le \gamma/\beta,\; y=\mu.
\end{cases}
$$
Note that 
\begin{align}
w+\beta xy \partial_xw+(\gamma-\beta x)y\partial_yw = x+y+\gamma y +\frac{g(x)+\beta x y g'(x)}{y-\mu}+\beta y\frac{(x-\gamma/\beta)g(x)}{(y-\mu)^2}
\end{align}
is a positive function in $(0,\infty)\times(\mu,\infty)$; this is due to our assumptions that $g(x)$, $g'(x)$, and $(x-\gamma/\beta)g(x)$ are all nonnegative. We conclude
\be\label{wDiffIneq}
w+\beta xy \partial_xw+(\gamma-\beta x)y\partial_yw\ge 0
\ee
in $(0,\infty)\times(\mu,\infty)$. 

\par Let us consider  the quantity
\be\label{emmm}
m:=\sup\left\{v^1(x,y)-v^2(x,y)-\epsilon w(x,y): x\ge 0,\;y>\mu \right\}
\ee
for a given $\epsilon>0$. We claim that 
\be\label{mepsilonInequality}
m\le 0.
\ee
This would in turn imply that $v^1\le v^2+\epsilon w$ for all $\epsilon>0$, and therefore $v^1\le v^2$. Likewise, we would have 
$v^2\le v^1$.  Consequently, we will focus on verifying inequality \eqref{mepsilonInequality}.
 
\par To this end, we note that since $v^1$ is bounded from above and that $v^2$ and $w$ are positive, $m$ is finite.  And as 
\be\label{doubleLim1}
\lim_{\substack{x+y\rightarrow\infty \\ y>\mu}}w(x,y)=\infty
\ee
 and 
 \be\label{doubleLim2}
 \lim_{\substack{(x,y)\rightarrow(x^0,\mu) \\ y>\mu}}w(x,y)=\infty
 \ee
 for each $x^0>\gamma/\beta$, there is $(\hat{x},\hat{y})\in [0,\infty)\times[\mu,\infty)$ for which  
 $$
 m=v^1(\hat{x},\hat{y})-v^2(\hat{x},\hat{y})-\epsilon w(\hat{x},\hat{y}).
 $$
 In particular, if $\hat{y}=\mu$ then $\hat{x}\in [0,\gamma/\beta]$. In this case, $v^1(\hat{x},\hat{y})=v^2(\hat{x},\hat{y})$ by \eqref{BCsegmentv} so  \eqref{mepsilonInequality} holds. 
Using  \eqref{BCsegment2v}, we can similarly conclude that \eqref{mepsilonInequality} holds if $\hat{x}=0$.

\par Now suppose that $\hat{x}>0$ and $\hat{y}>\mu$.  Our hypothesis that $v^1$ and $v^2$ are continuously differentiable gives 
$$
\begin{cases}
0=\partial_xv^1(\hat{x},\hat{y})-\partial_xv^2(\hat{x},\hat{y})-\epsilon \partial_xw(\hat{x},\hat{y})\\\\
0=\partial_yv^1(\hat{x},\hat{y})-\partial_yv^2(\hat{x},\hat{y})-\epsilon \partial_yw(\hat{x},\hat{y}).
\end{cases}
$$
In particular, we note that since $ \partial_xw(\hat{x},\hat{y})\ge 0$ 
\begin{align*}
(\partial_xv^1(\hat{x},\hat{y}))^-&=\max\{-\partial_xv^1(\hat{x},\hat{y}),0\}\\ 
&=\max\{-\partial_xv^2(\hat{x},\hat{y})-\epsilon \partial_xw(\hat{x},\hat{y}),0\}\\
&\le \max\{-\partial_xv^2(\hat{x},\hat{y}),0\}\\
&=(\partial_xv^2(\hat{x},\hat{y}))^-.
\end{align*}
\par Since $v^1$ and $v^2$ are solutions of \eqref{HJBvee} and $w$ satisfies \eqref{wDiffIneq},
\begin{align*}
m&=v^1(\hat{x},\hat{y})-v^2(\hat{x},\hat{y})-\epsilon w(\hat{x},\hat{y})\\
&=-\beta \hat{x} \hat{y}\partial_xv^1(\hat{x},\hat{y})+\hat{x}(v^1(\hat{x},\hat{y}))^-
-(\gamma-\beta \hat{x})\hat{y} v^1(\hat{x},\hat{y})\\
&\quad\quad +\beta \hat{x} \hat{y}\partial_xv^2(\hat{x},\hat{y})-\hat{x}(v^2(\hat{x},\hat{y}))^-
+(\gamma-\beta \hat{x})\hat{y} v^2(\hat{x},\hat{y})-\epsilon w(\hat{x},\hat{y})\\
&=-\beta \hat{x} \hat{y} (\partial_xv^1(\hat{x},\hat{y})-\partial_xv^2(\hat{x},\hat{y})) \\
&\quad - (\gamma-\beta \hat{x})\hat{y}(\partial_yv^1(\hat{x},\hat{y})-\partial_yv^2(\hat{x},\hat{y}))\\
&\quad +\hat{x}[(\partial_xv^1(\hat{x},\hat{y}))^--(\partial_xv^2(\hat{x},\hat{y}))^-]-\epsilon w(\hat{x},\hat{y})\\
&=-\epsilon\left[w(\hat{x},\hat{y})+\beta \hat{x} \hat{y}\partial_xw(\hat{x},\hat{y})+ (\gamma-\beta \hat{x})\hat{y}\partial_yw(\hat{x},\hat{y})\right]\\
&\quad +\hat{x}[(\partial_xv^1(\hat{x},\hat{y}))^--(\partial_xv^2(\hat{x},\hat{y}))^-]\\
&\le 0.
\end{align*}
Therefore, we conclude that \eqref{mepsilonInequality} holds in all cases. 
 \end{proof}
 Now we will issue a proof of Proposition \ref{VeeProp} without assuming $v^1,v^2$ are continuously differentiable. Again we emphasize that this proposition and Corollary \ref{ViscProperty} imply Theorem \ref{thm0}.
 
 \begin{proof}[Proof of Proposition \ref{VeeProp}] 
 1. We will fix $\epsilon>0$ and use the same notation as in the proof of this assertion in the special case that $v^1,v^2$ are continuously differentiable.  In particular, our goal is to show that $m$ defined in \eqref{emmm} is nonpositive.  Accordingly, we set
\begin{align*}
m_{\alpha}:&=\sup\bigg\{v^1(x_1,y_1)-v^2(x_2,y_2)-\frac{\epsilon}{2}(w(x_1,y_1)+w(x_2,y_2))   \\ 
&\quad\quad -\frac{1}{2\alpha}((x_1-x_2)^2+(y_1-y_2)^2): x_1,x_2\ge 0, y_1,y_2>\mu\bigg\}
\end{align*}
for $\alpha>0$.  It is not hard to see that  
\be\label{malphaBounded}
-\infty<m\le m_{\alpha}\le \sup v^1<\infty
\ee
for each $\alpha>0$. 

\par In addition, note that for any pairs $(x_1,y_1),(x_2,y_2)\in [0,\infty)\times(\mu,\infty)$
\begin{align}
&v^1(x_1,y_1)-v^2(x_2,y_2)-\frac{\epsilon}{2}(w(x_1,y_1)+w(x_2,y_2))-\frac{1}{2\alpha}((x_1-x_2)^2+(y_1-y_2)^2)\\
&\quad \le \sup v^1-\frac{\epsilon}{2}(w(x_1,y_1)+w(x_2,y_2)).
\end{align}
This inequality combined with \eqref{doubleLim1} and \eqref{doubleLim2} implies the existence of $(x^\alpha_1,y^\alpha_1),(x^\alpha_2,y^\alpha_2)\in [0,\infty)\times[\mu,\infty)$ such that 
\be\label{malphaIdentity}
m_{\alpha}=v^1(x^\alpha_1,y^\alpha_1)-v^2(x^\alpha_2,y^\alpha_2)-\frac{\epsilon}{2}(w(x^\alpha_1,y^\alpha_1)+w(x^\alpha_2,y^\alpha_2))    -\frac{1}{2\alpha}((x^\alpha_1-x^\alpha_2)^2+(y^\alpha_1-y^\alpha_2)^2)
\ee
and
\be\label{alphaSeqBounded}
\sup_{\alpha>0}(x^\alpha_1+y^\alpha_1 +x^\alpha_2+y^\alpha_2)<\infty.
\ee
Furthermore, if $y^\alpha_1=\mu$, then $x^\alpha_1\in [0,\gamma/\beta]$; and if $y^\alpha_2=\mu$, then $x^\alpha_2\in [0,\gamma/\beta]$.
Proposition 3.7 of \cite{MR1118699} also implies 
\be\label{DoublingVarLim}
\lim_{\alpha\rightarrow0^+}\frac{1}{2\alpha}((x^\alpha_1-x^\alpha_2)^2+(y^\alpha_1-y^\alpha_2)^2)=0
\ee
and 
$$
\lim_{\alpha\rightarrow0^+}m_\alpha=m.
$$

\par 2. In view of \eqref{alphaSeqBounded}, we may select a sequence of positive numbers $\alpha_k$ tending to $0$ as $k\rightarrow\infty$ so that 
$$
\hat x:=\lim_{k\rightarrow\infty}x^{\alpha_k}_1=\lim_{k\rightarrow\infty}x^{\alpha_k}_2
$$
and
$$
\hat y:=\lim_{k\rightarrow\infty}y^{\alpha_k}_1=\lim_{k\rightarrow\infty}y^{\alpha_k}_2.
$$
If $\hat x=0$, then 
$$
m=\lim_{k\rightarrow\infty}m_{\alpha_k}=v^1(0,\hat{y})-v^2(0,\hat{y})-\epsilon w(0,\hat{y})\le 0
$$
by \eqref{BCsegment2v}.  Likewise, if $\hat y=\mu$ 
\begin{align}
m&=\lim_{k\rightarrow\infty}m_{\alpha_k}\\
&=\lim_{k\rightarrow\infty}\bigg(v^1(x^{\alpha_k}_1,y^{\alpha_k}_1)-v^2(x^{\alpha_k}_2,y^{\alpha_k}_2)-\frac{\epsilon}{2}(w(x^{\alpha_k}_1,y^{\alpha_k}_1)+w(x^{\alpha_k}_2,y^{\alpha_k}_2)) \\
&\quad   -\frac{1}{2{\alpha_k}}((x^{\alpha_k}_1-x^{\alpha_k}_2)^2+(y^{\alpha_k}_1-y^{\alpha_k}_2)^2)\bigg) \\
&=v^1(\hat x,\mu)-v^2(\hat{x},\mu)-\frac{\epsilon}{2} \lim_{k\rightarrow\infty}(w(x^{\alpha_k}_1,y^{\alpha_k}_1)+w(x^{\alpha_k}_2,y^{\alpha_k}_2)).
\end{align}
As $(w(x^{\alpha_k}_1,y^{\alpha_k}_1))_{k\in \N}$ and $(w(x^{\alpha_k}_1,y^{\alpha_k}_1))_{k\in \N}$ are bounded from above, it must be that $0\le \hat x\le\gamma/\beta$. This in turn implies 
$$
m\le v^1(\hat x,\mu)-v^2(\hat{x},\mu)=0
$$
by \eqref{BCsegmentv}.
 
 \par 3. Alternatively, $(\hat x,\hat y)\in (0,\infty)\times (\mu, \infty)$. In this case, $(x^{\alpha_k}_1, y^{\alpha_k}_1), (x^{\alpha_k}_2, y^{\alpha_k}_2)\in (0,\infty)\times(\mu,\infty)$ for all sufficiently large $k\in \N$. In view of \eqref{malphaIdentity}, the function of $(x_1, y_1)$
 $$
v^1(x_1,y_1)-\left[v^2(x^{\alpha_k}_2, y^{\alpha_k}_2)+\frac{\epsilon}{2}(w(x_1,y_1)+w(x^{\alpha_k}_2, y^{\alpha_k}_2))+\frac{1}{2\alpha_k}((x_1-x^{\alpha_k}_2)^2+(y_1-y^{\alpha_k}_2)^2)\right]
$$
has a maximum at $(x^{\alpha_k}_1, y^{\alpha_k}_1)$.  Since $v^1$ is a viscosity solution of \eqref{HJBvee}, 
\begin{align}\label{v1viscIneq}
&v^1(x^{\alpha_k}_1, y^{\alpha_k}_1)+\beta x^{\alpha_k}_1y^{\alpha_k}_1\left[\frac{\epsilon}{2}\partial_xw(x^{\alpha_k}_1, y^{\alpha_k}_1)+\frac{x^{\alpha_k}_1-x^{\alpha_k}_2}{\alpha_k}\right]\\
&\quad -x^{\alpha_k}_1\left[\frac{\epsilon}{2}\partial_xw(x^{\alpha_k}_1, y^{\alpha_k}_1)+\frac{x^{\alpha_k}_1-x^{\alpha_k}_2}{\alpha_k}\right]^-\\
&\quad+(\gamma-\beta x^{\alpha_k}_1)y^{\alpha_k}_1\left[\frac{\epsilon}{2}\partial_yw(x^{\alpha_k}_1, y^{\alpha_k}_1)+\frac{y^{\alpha_k}_1-y^{\alpha_k}_2}{\alpha_k}\right]\le0.
\end{align}

\par Likewise the function of $(x_2, y_2)$
  $$
v^2(x_2,y_2)-\left[v^1(x^{\alpha_k}_1, y^{\alpha_k}_1)-\frac{\epsilon}{2}(w(x^{\alpha_k}_1, y^{\alpha_k}_1)+w(x_2,y_2))-\frac{1}{2\alpha_k}((x_2-x^{\alpha_k}_1)^2+(y_2-y^{\alpha_k}_1)^2)\right]
$$
has a minimum at $(x^{\alpha_k}_2, y^{\alpha_k}_2)$. As $v^2$ is a viscosity solution of \eqref{HJBvee}, 
\begin{align}\label{v2viscIneq}
&v^2(x^{\alpha_k}_2, y^{\alpha_k}_2)+\beta x^{\alpha_k}_2y^{\alpha_k}_2\left[-\frac{\epsilon}{2}\partial_xw(x^{\alpha_k}_2, y^{\alpha_k}_2)-\frac{x^{\alpha_k}_2-x^{\alpha_k}_1}{\alpha_k}\right]\\
&\quad -x^{\alpha_k}_2\left[-\frac{\epsilon}{2}\partial_xw(x^{\alpha_k}_2, y^{\alpha_k}_2)-\frac{x^{\alpha_k}_2-x^{\alpha_k}_1}{\alpha_k}\right]^-\\
&\quad+(\gamma-\beta x^{\alpha_k}_2)y^{\alpha_k}_2\left[-\frac{\epsilon}{2}\partial_yw(x^{\alpha_k}_2, y^{\alpha_k}_2)-\frac{y^{\alpha_k}_2-y^{\alpha_k}_1}{\alpha_k}\right]\ge0.
\end{align}
We can then combine \eqref{v1viscIneq} and \eqref{v2viscIneq} to get 
\begin{align}\label{v1plusv2viscIneq}
m_{\alpha_k}&\le v^1(x^{\alpha_k}_1, y^{\alpha_k}_1) -v^2(x^{\alpha_k}_2, y^{\alpha_k}_2)-\frac{\epsilon}{2}(w(x^{\alpha_k}_1, y^{\alpha_k}_1)+w(x^{\alpha_k}_2, y^{\alpha_k}_2)) \\
&\le \beta x^{\alpha_k}_2y^{\alpha_k}_2\left[-\frac{\epsilon}{2}\partial_xw(x^{\alpha_k}_2, y^{\alpha_k}_2)-\frac{x^{\alpha_k}_2-x^{\alpha_k}_1}{\alpha_k}\right]-\beta x^{\alpha_k}_1y^{\alpha_k}_1\left[\frac{\epsilon}{2}\partial_xw(x^{\alpha_k}_1, y^{\alpha_k}_1)+\frac{x^{\alpha_k}_1-x^{\alpha_k}_2}{\alpha_k}\right]\\
&\quad + x^{\alpha_k}_1\left[\frac{\epsilon}{2}\partial_xw(x^{\alpha_k}_1, y^{\alpha_k}_1)+\frac{x^{\alpha_k}_1-x^{\alpha_k}_2}{\alpha_k}\right]^-
-x^{\alpha_k}_2\left[-\frac{\epsilon}{2}\partial_xw(x^{\alpha_k}_2, y^{\alpha_k}_2)-\frac{x^{\alpha_k}_2-x^{\alpha_k}_1}{\alpha_k}\right]^-\\
&\quad +(\gamma-\beta x^{\alpha_k}_2)y^{\alpha_k}_2\left[-\frac{\epsilon}{2}\partial_yw(x^{\alpha_k}_2, y^{\alpha_k}_2)-\frac{y^{\alpha_k}_2-y^{\alpha_k}_1}{\alpha_k}\right] \\ 
&\quad -(\gamma-\beta x^{\alpha_k}_1)y^{\alpha_k}_1\left[\frac{\epsilon}{2}\partial_yw(x^{\alpha_k}_1, y^{\alpha_k}_1)+\frac{y^{\alpha_k}_1-y^{\alpha_k}_2}{\alpha_k}\right] -\frac{\epsilon}{2}(w(x^{\alpha_k}_1, y^{\alpha_k}_1)+w(x^{\alpha_k}_2, y^{\alpha_k}_2)).
\end{align}

\par 4. We will now proceed to estimate some of the terms on the right hand side of \eqref{v1plusv2viscIneq}. First, observe 
\begin{align}\label{v1plusv2viscIneq2}
&-\beta x^{\alpha_k}_2y^{\alpha_k}_2\cdot \frac{x^{\alpha_k}_2-x^{\alpha_k}_1}{\alpha_k}-\beta x^{\alpha_k}_1y^{\alpha_k}_1\cdot \frac{x^{\alpha_k}_1-x^{\alpha_k}_2}{\alpha_k}\\
&\quad = \beta\frac{x^{\alpha_k}_1-x^{\alpha_k}_2}{\alpha_k}\left[  x^{\alpha_k}_2y^{\alpha_k}_2- x^{\alpha_k}_1y^{\alpha_k}_1 \right]\\
&\quad = \beta\frac{x^{\alpha_k}_1-x^{\alpha_k}_2}{\alpha_k}\left[  x^{\alpha_k}_2y^{\alpha_k}_2- x^{\alpha_k}_1y^{\alpha_k}_2+x^{\alpha_k}_1y^{\alpha_k}_2-x^{\alpha_k}_1y^{\alpha_k}_1 \right]\\
&\quad = -\beta\frac{(x^{\alpha_k}_2-x^{\alpha_k}_1)^2}{\alpha_k}y^{\alpha_k}_2+\beta\frac{x^{\alpha_k}_1-x^{\alpha_k}_2}{\sqrt{\alpha_k}}\frac{y^{\alpha_k}_2-y^{\alpha_k}_1}{\sqrt{\alpha_k}}x^{\alpha_k}_1\\
&=o(1)
\end{align}
as $k\rightarrow\infty$ by \eqref{DoublingVarLim}. Similarly,
\begin{align}\label{v1plusv2viscIneq3}
&-(\gamma-\beta x^{\alpha_k}_2)y^{\alpha_k}_2\frac{y^{\alpha_k}_2-y^{\alpha_k}_1}{\alpha_k}- (\gamma-\beta x^{\alpha_k}_1)y^{\alpha_k}_1\frac{y^{\alpha_k}_1-y^{\alpha_k}_2}{\alpha_k}\\
&\quad= \frac{y^{\alpha_k}_1-y^{\alpha_k}_2}{\alpha_k}\left((\gamma-\beta x^{\alpha_k}_2)y^{\alpha_k}_2- (\gamma-\beta x^{\alpha_k}_1)y^{\alpha_k}_1\right)\\
&\quad= -\gamma\frac{(y^{\alpha_k}_1-y^{\alpha_k}_2)^2}{\alpha_k}+\beta\frac{y^{\alpha_k}_1-y^{\alpha_k}_2}{\alpha_k}\left( x^{\alpha_k}_1 y^{\alpha_k}_1- x^{\alpha_k}_2 y^{\alpha_k}_2\right)\\
&\quad =-(\gamma-\beta x_1^{\alpha_k})\frac{(y^{\alpha_k}_1-y^{\alpha_k}_2)^2}{\alpha_k}+\beta y^{\alpha_k}_2\frac{x^{\alpha_k}_1-x^{\alpha_k}_2}{\sqrt{\alpha_k}}\frac{y^{\alpha_k}_1-y^{\alpha_k}_2}{\sqrt{\alpha_k}}\\
&=o(1)
\end{align}
as $k\rightarrow\infty$. 

\par Also notice 
\begin{align}
&x^{\alpha_k}_1\left[\frac{\epsilon}{2}\partial_xw(x^{\alpha_k}_1, y^{\alpha_k}_1)+\frac{x^{\alpha_k}_1-x^{\alpha_k}_2}{\alpha_k}\right]^-
-x^{\alpha_k}_2\left[-\frac{\epsilon}{2}\partial_xw(x^{\alpha_k}_2, y^{\alpha_k}_2)-\frac{x^{\alpha_k}_2-x^{\alpha_k}_1}{\alpha_k}\right]^- \\
&\quad = (x^{\alpha_k}_1-x^{\alpha_k}_2)\left[\frac{\epsilon}{2}\partial_xw(x^{\alpha_k}_1, y^{\alpha_k}_1)+\frac{x^{\alpha_k}_1-x^{\alpha_k}_2}{\alpha_k}\right]^-+\\
&\quad\quad x^{\alpha_k}_2\left(\left[\frac{\epsilon}{2}\partial_xw(x^{\alpha_k}_1, y^{\alpha_k}_1)+\frac{x^{\alpha_k}_1-x^{\alpha_k}_2}{\alpha_k}\right]^-
-\left[-\frac{\epsilon}{2}\partial_xw(x^{\alpha_k}_2, y^{\alpha_k}_2)-\frac{x^{\alpha_k}_2-x^{\alpha_k}_1}{\alpha_k}\right]^-\right)\\
&\quad = o(1)+x^{\alpha_k}_2\left(\left[\frac{\epsilon}{2}\partial_xw(x^{\alpha_k}_1, y^{\alpha_k}_1)+\frac{x^{\alpha_k}_1-x^{\alpha_k}_2}{\alpha_k}\right]^-
-\left[-\frac{\epsilon}{2}\partial_xw(x^{\alpha_k}_2, y^{\alpha_k}_2)-\frac{x^{\alpha_k}_2-x^{\alpha_k}_1}{\alpha_k}\right]^-\right)
\end{align}
 as $k\rightarrow\infty$.  Employing the elementary inequality 
$$
a^--b^-\le (a-b)^-\quad (a,b\in \R),
$$
we then find 
\begin{align}\label{v1plusv2viscIneq4}
&x^{\alpha_k}_1\left[\frac{\epsilon}{2}\partial_xw(x^{\alpha_k}_1, y^{\alpha_k}_1)+\frac{x^{\alpha_k}_1-x^{\alpha_k}_2}{\alpha_k}\right]^-
-x^{\alpha_k}_2\left[-\frac{\epsilon}{2}\partial_xw(x^{\alpha_k}_2, y^{\alpha_k}_2)-\frac{x^{\alpha_k}_2-x^{\alpha_k}_1}{\alpha_k}\right]^-\\
&\le  o(1)+x^{\alpha_k}_2\left[\frac{\epsilon}{2}\left(\partial_xw(x^{\alpha_k}_1, y^{\alpha_k}_1)+\partial_xw(x^{\alpha_k}_2, y^{\alpha_k}_2)\right)\right]^-\\
&\le  o(1)
\end{align} 
 as $k\rightarrow\infty$. Here we used $\partial_xw\ge 0$. 
 
 \par In view of \eqref{v1plusv2viscIneq2}, \eqref{v1plusv2viscIneq3}, and \eqref{v1plusv2viscIneq4}, we can send  $k\rightarrow\infty$ in \eqref{v1plusv2viscIneq} to get 
 \begin{align}
 m\le - \epsilon\left[w(\hat{x},\hat{y})+\beta \hat{x} \hat{y}\partial_xw(\hat{x},\hat{y})+ (\gamma-\beta \hat{x})\hat{y}\partial_yw(\hat{x},\hat{y})\right]
 \le 0
 \end{align}
 by \eqref{wDiffIneq}. As mentioned at the start of this proof, the proposition in question follows from the nonpositivity of $m$.  Therefore, we have obtained the desired conclusion.
   \end{proof}

\section{Local semiconcavity}\label{SemiConcaveSect}
Let us now investigate the differentiability of the value function. We'll argue that $u$ is twice differentiable almost everywhere and its Hessian is locally bounded above.  We will establish these properties by deriving various bounds on $u^r$ that are independent of $r\in {\cal A}$.  
With these goals in mind, we will study $u^r$ and $u$ on triangular domains 
$$
T_{N,\delta}:=\{(x,y)\in \R^2: x\ge \delta,y\ge \mu+\delta,\;\text{and}\; x+y\le N\}
$$
for $N,\delta\ge 0$ which satisfy
$$
\mu+2\delta<N.
$$

\par We will also employ the {\it flow} of the controlled SIR system \eqref{SIRcontrol}
$$
\Phi^r: [0,\infty)\times[\mu,\infty)\times[0,\infty)\rightarrow [0,\infty)^2; (x,y,t)\mapsto (\Phi^r_1(x,y,t),\Phi^r_2(x,y,t)).
$$
Here $S^r(t)=\Phi^r_1(x,y,t)$ and $I^r(t)=\Phi^r_2(x,y,t)$ is the solution of the controlled SIR system \eqref{SIRcontrol} with 
$S^r(0)=x$ and $I^r(0)=y$.  We note that for any $r\in {\cal A}$, 
$$
(x,y)\mapsto \Phi^r(x,y,t)\;\text{is smooth for each $t\ge 0$} 
$$
by Theorem 3.3 and Exercise 3.2 of \cite{MR587488}, and 
$$
(x,y,t)\mapsto \Phi^r_2(x,y,t)\;\text{is continuously differentiable}
$$
by Lemma \ref{ExistenceLemmaSIRcontrol}.

\par It also follows that $(x,y,t)\mapsto \Phi^r(x,y,t)$ is smooth whenever $r$ is smooth. Since $u^r$ satisfies the implicit 
equation
\be\label{impliciturequation}
\Phi_2^r(x,y,u^r(x,y))=\mu
\ee
for each $x>0$ and $y>\mu$, we can then differentiate this equation twice to obtain bounds on the second derivatives of $u^r$ when $r$ is smooth.  Of course, $r$ is in general not smooth. We will get around this by finding estimates which are independent of $r$ and using the fact that $r\mapsto u^r$ is continuous.

\par  To this end, we will employ the following assertion about solutions of linear ODEs.  Since this claim follows from an easy application of Gr\"onwall's lemma, we will omit the proof. 
\begin{lem}
Suppose $A(t)$ is a $2\times 2$ matrix for each $t\ge0$ and $f: [0,\infty)\rightarrow \R^2$.  If $z: [0,\infty)\rightarrow \R^2$ solves 
\be\label{zODE}
\dot z(t)=A(t)z(t)+f(t), \quad t\ge 0,
\ee
then
$$
\|z(t)\|^2\le e^{(2c+1)t}\left(\|z(0)\|^2 + \int^t_0e^{-(2c+1)s} \|f(s)\|^2ds\right)
$$
where $c\ge \max_{t\ge 0}\|A(t)\|$.
\end{lem}
\begin{rem}
Here $\|A(t)\|$ denotes the Frobenius norm of $A(t)$.
\end{rem}
We will now derive various bounds on the derivatives of $\Phi_2^r$ when $r$ is smooth.  
\begin{lem}\label{DerivativeBoundsPhi2}
Suppose $r\in {\cal A}$ is smooth and $N> \mu$. Then 
\begin{enumerate}[(i)]

\item $0\le \Phi^r_2(x,y,t)\le N$
\\
\item $|\partial_t\Phi^r_2(x,y,t)|\le \beta N^2+\gamma N$
\\
\item $|\partial_x\Phi^r_2(x,y,t)|,|\partial_y\Phi^r_2(x,y,t)|\le\displaystyle e^{(C+1/2) t}$
\\
\item $|\partial_t^2\Phi^r_2(x,y,t)|\le  (\beta N^2+N)\beta N+(\beta^2N^2+\gamma^2)N$
\\
\item $|\partial_x\partial_t\Phi^r_2(x,y,t)|,|\partial_y\partial_t\Phi^r_2(x,y,t)|\le \displaystyle \displaystyle e^{(C+1/2)t}(2\beta N+\gamma)$
\\
\item $|\partial^2_x\Phi^r_2(x,y,t)|,|\partial^2_y\Phi^r_2(x,y,t)|,|\partial_x\partial_y\Phi^r_2(x,y,t)|\le \displaystyle\frac{2\beta}{\sqrt{C}}e^{(2C+1) t}$

\end{enumerate}
for $(x,y)\in T_{0, N}$ and $t\ge 0$.  Here $C:=\sqrt{6\beta^2 N^2+2\gamma^2+2}.$
\end{lem}
\begin{proof}
Assertions $(i),(ii)$ and $(iv)$ follow from the proof of Lemma \ref{ExistenceLemmaSIRcontrol}. Moreover, 
\be\label{propertyzero}
0\le \Phi^r_1(x,y,t)\le N
\ee
holds for $(x,y)\in T_{0, N}$, as well. Differentiating \eqref{SIRcontrol} with respect to $x$ gives
\be\label{linearizationfirstinx}
\partial_t\left(\begin{array}{c}
 \partial_x\Phi^r_1(x,y,t)  \\
 \partial_x\Phi^r_2(x,y,t)  
\end{array}\right)
=
A(t)\left(\begin{array}{c}
 \partial_x\Phi^r_1(x,y,t)  \\
 \partial_x\Phi^r_2(x,y,t)  
\end{array}\right),
\ee
with 
$$
A(t)=\left(\begin{array}{cc}
-\beta \Phi^r_2(x,y,t) -r(t) & -\beta \Phi^r_1(x,y,t) \\
\beta \Phi^r_2(x,y,t) & \beta \Phi^r_1(x,y,t) -\gamma
\end{array}\right).
$$ 
As
\begin{align}
\|A(t)\|^2&=(\beta \Phi^r_2(x,y,t) +r(t))^2+(\beta \Phi^r_1(x,y,t) )^2+(\beta \Phi^r_2(x,y,t))^2 +(\beta \Phi^r_1(x,y,t) -\gamma)^2\\
&\le 2(\beta^2N^2+1)+\beta^2N^2+\beta^2N^2+2(\beta^2N^2+\gamma^2)\\
&=6\beta^2 N^2+2\gamma^2+2\\
&=C^2
\end{align}
and $\|\partial_x\Phi^r(x,y,0)\|=1$, the previous lemma with $z(t)=\partial_x\Phi^r(x,y,t)$ implies
\be\label{gradxPhir2bound}
\|\partial_x\Phi^r(x,y,t)\|\le e^{(C+1/2)t}.
\ee
We can derive the same upper bound for $\|\partial_y\Phi^r(x,y,t)\|$, so assertion $(iii)$ follows. 

\par The second equation in \eqref{linearizationfirstinx} is 
$$
\partial_t\partial_x\Phi^r_2(x,y,t)=\beta (\partial_x\Phi^r_1(x,y,t)\Phi^r_2(x,y,t)+\Phi^r_1(x,y,t)\partial_x\Phi^r_2(x,y,t)) -\gamma \partial_x\Phi^r_2(x,y,t).
$$
Using $(i)$ and \eqref{gradxPhir2bound} leads to 
$$
|\partial_t\partial_x\Phi^r_2(x,y,t)|\le  e^{(C+1/2) t}(2\beta N+\gamma).
$$
This method also leads to the same bound for $|\partial_t\partial_y\Phi^r_2(x,y,t)|$.  This proves $(v)$.

\par As for $(vi)$, we will prove the estimate for $\partial^2_x\Phi^r_2(x,y,t)$ as the other cases can be handled similarly.  Upon differentiating 
 \eqref{linearizationfirstinx} with respect to $x$, 
 \be\label{Secondlinearizationinx}
\partial_t\left(\begin{array}{c}
 \partial^2_x\Phi^r_1(x,y,t)  \\
 \partial^2_x\Phi^r_2(x,y,t)  
\end{array}\right)
=
A(t)\left(\begin{array}{c}
 \partial^2_x\Phi^r_1(x,y,t)  \\
 \partial^2_x\Phi^r_2(x,y,t)  
\end{array}\right)+2\beta\left(\begin{array}{c}
 -\partial_x\Phi^r_1(x,y,t)\partial_x\Phi^r_2(x,y,t)  \\
\partial_x\Phi^r_1(x,y,t)\partial_x\Phi^r_2(x,y,t) 
\end{array}\right).
\ee
Note that $\partial^2_x\Phi^r_1(x,y,0) =\partial^2_x\Phi^r_2(x,y,0) =0$ and 
$$
\left\|2\beta\left(\begin{array}{c}
 -\partial_x\Phi^r_1(x,y,t)\partial_x\Phi^r_2(x,y,t)  \\
\partial_x\Phi^r_1(x,y,t)\partial_x\Phi^r_2(x,y,t) 
\end{array}\right)\right\|^2\le8\beta^2 e^{4(C+1/2)t}
$$
by \eqref{gradxPhir2bound}. The previous lemma then implies 
\begin{align}
| \partial^2_x\Phi^r_2(x,y,t) |^2&\le  e^{(2C+1)t}8\beta^2\int^t_0e^{-(2C+1)s} e^{(4C+2)s}ds\\
&\le  e^{(2C+1)t}8\beta^2\int^t_0e^{(2C+1)s}ds\\
&\le  e^{(2C+1)t}8\beta^2\frac{1}{2C+1}(e^{(2C+1)t}-1)\\
&\le e^{(4C+2)t}\frac{4\beta^2}{C}.
\end{align}
\end{proof}
When we differentiate \eqref{impliciturequation}, it will be crucial that 
$$
\partial_t\Phi^r_2(x,y,t)= (\beta\Phi^r_1(x,y,t) -\gamma) \Phi^r_2(x,y,t)< 0
$$
when $t=u^r(x,y)$. We'll show that this quantity is uniformly bounded away from 0 for $(x,y)\in T_{N,\delta}$.
\begin{lem}
Suppose $N,\delta>0$ satisfy $N>\mu+2\delta$. There is 
$$
\epsilon >0
$$
such that 
\be\label{NondegdotIarr}
\gamma-\beta\Phi^r_1(x,y,u^r(x,y)) \ge \epsilon
\ee
for each $r\in {\cal A}$ and $(x,y)\in T_{N,\delta}$.
\end{lem}
\begin{proof}
If this assertion is false, for each $k\in \N$ there would be $(x^k,y^k)\in T_{N,\delta}$ and $r^k\in {\cal A}$ such that 
$$
\gamma-\beta\Phi^{r^k}_1(x^k,y^k,u^{r^k}(x^k,y^k))<\frac{1}{k}.
$$
Passing to subsequences if necessary, we may suppose that $(x^k,y^k)\rightarrow (x,y) \in T_{N,\delta}$ and $r^k\rightarrow r$ weak$^*$ in ${\cal A}$.   According to Proposition \ref{CompactnessProp} and \eqref{Convergenceurkayxkayykay}, 
$$
\gamma-\beta\Phi^{r}_1(x,y,u^{r}(x,y))\le 0.
$$
However, this contradicts Lemma \ref{Sulessthan} as $y\ge \mu+2\delta>\mu$.
\end{proof}
\begin{cor}\label{uarrSmoothCor}
For each $r\in {\cal A}$, $u^r$ is continuously differentiable on $(0,\infty)\times(\mu,\infty)$.  Moreover, for each $N,\delta>0$ with $N>\mu+2\delta$, there is a constant 
$B$ such that 
\be\label{LipschitzBounduArr}
|\partial_xu^r(x,y)|,|\partial_yu^r(x,y)|\le B
\ee
for $(x,y)\in T_{N,\delta}$ and $r\in {\cal A}$.
\end{cor}
\begin{proof}
Recall that $\Phi^r_2$ is continuously differentiable for all $r\in {\cal A}$.  In view of the previous lemma,
$$
\partial_t\Phi^r_2(x,y,u^{r}(x,y))<0
$$ 
for $x>0$ and $y>\mu$.  
Applying the implicit function theorem to equation \eqref{impliciturequation} gives that $u^r$ is continuously differentiable. 

\par Differentiating \eqref{impliciturequation} with respect to $x$ gives 
\be\label{impliciturequation2} 
\partial_x\Phi^r_2(x,y,u^{r}(x,y))+\partial_t\Phi^r_2(x,y,u^{r}(x,y))\partial_xu^{r}(x,y)=0.
\ee
We can then use the previous lemma, part $(iii)$ of Lemma \ref{DerivativeBoundsPhi2} and \eqref{uarrUpperBound} to find 
\be
|\partial_xu^{r}(x,y)|=\frac{|\partial_x\Phi^r_2(x,y,u^{r}(x,y))|}{|\partial_t\Phi^r_2(x,y,u^{r}(x,y))|}\le\frac{1}{\epsilon}e^{ (C+1/2) u^r(x,y)}\le\frac{1}{\epsilon}e^{ (C+1/2)N/\mu\gamma}
\ee
for $(x,y)\in T_{N,\delta}$. Here $\epsilon$ is the constant in \eqref{NondegdotIarr} and $C=\sqrt{6\beta^2 N^2+2\gamma^2+2}.$  The same upper bound also holds for $|\partial_yu^{r}(x,y)|$.
\end{proof}
\begin{rem}\label{Extenduxuy}
Using \eqref{impliciturequation2} and the corresponding equation obtained by differentiating \eqref{impliciturequation} with respect to $y$, we can extend $\partial_xu^r$ and $\partial_yu^r$ continuously to the interval
$$
0<x<\frac{\gamma}{\beta}\quad \text{and} \quad y=\mu.
$$
In particular, 
\be 
\partial_xu^{r}(x,\mu)=0\quad\text{and}\quad \partial_yu^{r}(x,\mu)=\frac{1}{(\gamma-\beta x)\mu}
\ee
for any such pair $(x,y)$.
\end{rem}
In view of \eqref{LipschitzBounduArr} and the fact that $T_{N,\delta}$ is convex, 
$$
|u^r(x_1,y_1)-u^r(x_2,y_2)|\le \sqrt{2}B(|x_1-y_1|+|x_2-y_2|)
$$
for all $r\in {\cal A}$ and $(x_1,y_1),(x_2,y_2)\in T_{N,\delta}$.  As a result, 
$$
|u(x_1,y_1)-u(x_2,y_2)|\le \sqrt{2} B(|x_1-y_1|+|x_2-y_2|)
$$
for $(x_1,y_1),(x_2,y_2)\in T_{N,\delta}$.  In particular, the value function is Lipschitz continuous on $T_{N,\delta}$ for any $N>\mu+2\delta$. 
By Rademacher's theorem, $u$ is differentiable almost everywhere in $(0,\infty)\times(\mu,\infty)$.  

\par Now we will explain how to bound the second derivatives of $u^r$ when $r$ is smooth.   Our method involves differentiating equation \eqref{impliciturequation} twice.

\begin{prop}
Suppose $N,\delta>0$ satisfy $N>\mu+2\delta$. There is a constant $D$ such that
\be\label{SecondDerivativeBounds}
|\partial^2_xu^r(x,y)|,|\partial^2_yu^r(x,y)|,|\partial_x\partial_yu^r(x,y)|\le D
\ee
for each $(x,y)\in T_{N,\delta}$ and each smooth $r\in {\cal A}$. 
\end{prop}
\begin{proof}
We will establish the bound for $|\partial^2_xu^r(x,y)|$ as the other bounds can be similarly achieved.   Differentiating \eqref{impliciturequation2} 
with respect to $x$ gives 
\begin{align}
&0=\partial^2_x\Phi^r_2(x,y,u^{r}(x,y))+\partial_t\partial_x\Phi^r_2(x,y,u^{r}(x,y))\partial_xu^{r}(x,y)\\
&\quad+(\partial_x\partial_t\Phi^r_2(x,y,u^{r}(x,y))+\partial^2_t\Phi^r_2(x,y,u^{r}(x,y))\partial_xu^{r}(x,y) )\partial_xu^{r}(x,y)\\
&\quad+\partial_t\Phi^r_2(x,y,u^{r}(x,y))\partial^2_xu^{r}(x,y).
\end{align}
In view of  \eqref{uarrUpperBound}, $u^r(x,y)\le N/\mu\gamma$.  Consequently, 
$$
|\partial^2_x\Phi^r_2(x,y,u^{r}(x,y))|, |\partial_t\partial_x\Phi^r_2(x,y,u^{r}(x,y))|,|\partial^2_t\Phi^r_2(x,y,u^{r}(x,y))|
$$
are each uniformly bounded for all $(x,y)\in T_{N,\delta}$ independently of $r$. This follows from $(iv), (v)$, and $(vi)$ in  Lemma \ref{DerivativeBoundsPhi2}.  We can then solve the equation above for $\partial^2_xu^{r}(x,y)$ and use \eqref{NondegdotIarr} and \eqref{LipschitzBounduArr} to find $D$ as asserted. 
\end{proof}
\begin{cor}\label{uarrQuasiconcave}
Suppose $N,\delta>0$ satisfy $N>\mu+2\delta$.  Then there is a constant $L$ such 
$$
(x,y)\mapsto u^r(x,y)-\frac{L}{2}(x^2+y^2)
$$
is concave in $T_{N,\delta}$ for each $r\in {\cal A}$.
\end{cor}
\begin{proof}
First suppose $r\in {\cal A}$ is smooth. By \eqref{SecondDerivativeBounds}, we can find $L$ which is independent of $r$ such that 
$$
\left\|\left(\begin{array}{cc}
\partial^2_xu^r(x,y) & \partial_y \partial_xu^r(x,y) \\
\partial_x \partial_yu^r(x,y) & \partial^2_yu^r(x,y)
\end{array}\right)\right\|\le L
$$
for $(x,y)\in T_{N,\delta}$. It follows that the Hessian of 
$$
(x,y)\mapsto u^r(x,y)-\frac{L}{2}(x^2+y^2)
$$
is nonpositive definite in $T_{N,\delta}$, so this function is concave in $T_{N,\delta}$. 

\par Now suppose $r\in {\cal A}$ and choose a sequence of smooth  $r^k\in {\cal A}$ such that $r^k\rightarrow r$ weak$^*$. 
Such a sequence exists by standard smoothing techniques. See for example Appendix C.5 of \cite{MR2597943}. As 
$$
(x,y)\mapsto u^{r^k}(x,y)-\frac{L}{2}(x^2+y^2)
$$
is concave in $T_{N,\delta}$,
\be
u^{r^k}\left(\frac{x_1+x_2}{2},\frac{y_1+y_2}{2}\right)\ge \frac{1}{2}u^{r^k}(x_1,y_1)+\frac{1}{2}u^{r^k}(x_2,y_2)-L\left(\left(\frac{x_1-x_2}{2}\right)^2+\left(\frac{y_1-y_2}{2}\right)^2\right)
\ee
for each $(x_1,y_1), (x_2,y_2)\in T_{N,\delta}$. We can then send $k\rightarrow\infty$ in this inequality by \eqref{Convergenceurkayxkayykay} and conclude that this inequality holds for $u^r$. That is, $u^{r }(x,y)-\frac{L}{2}(x^2+y^2)$ is concave in $T_{N,\delta}$.
\end{proof}

\begin{proof}[Proof of Theorem \ref{thm1}]
Suppose $K\subset (0,\infty)\times(\mu,\infty)$ is convex and compact. Then $K\subset T_{N,\delta}$ for some $N,\delta>0$ with $N>\mu+2\delta$. Thus,
there is a constant $L$ such that  
$$
u(x,y)-\frac{L}{2}(x^2+y^2)=\inf_{r\in {\cal A}}\left(u^r(x,y)-\frac{L}{2}(x^2+y^2)\right)
$$
is necessarily concave in $K$ by Corollary \ref{uarrQuasiconcave}.  
\end{proof}

\section{Optimal switching}\label{OptSwtichSect}
In this section, we will prove Theorem \ref{thm2} and a few corollaries. Here we recall the definition of a switching time vaccination rate 
\be\label{SwitchingTimeControl2}
r_\tau(t)=
\begin{cases}
0, \quad t\in [0,\tau]\\
1,\quad t\in (\tau,\infty).
\end{cases}
\ee 
We will first need to make an elementary observation
\begin{lem}\label{restricttaulessthanuarrtau}
Suppose $x\ge 0$ and $y\ge \mu$. If 
$$
\rho:=u^{r_\tau}(x,y)<\tau,
$$
then
$$
\rho= u^{r_\rho}(x,y).
$$
\end{lem}
\begin{proof}
As $\rho<\tau$, $r_\rho(t)=r_\tau(t)$ for $t\in [0,\rho]$. It follows that $(S^{r_\rho}(t),I^{r_\rho}(t))=
(S^{r_\tau}(t),I^{r_\tau}(t))$ for $t\in [0,\rho]$.  In particular, $I^{r_\rho}(\rho)=I^{r_\tau}(\rho)=\mu$. Therefore,  
$$
u^{r_\rho}(x,y)\le \rho.
$$

\par Also note that if $y>\mu$ or if $0\le x\le \gamma/\beta$ and $y=\mu$, then $I^{r_\rho}(t)=\mu$ only has one solution $t=u^{r_\rho}(x,y)$. 
Alternatively, if $x>\gamma/\beta$ and $y=\mu$, then $I^{r_\rho}(t)=\mu$ has two solutions $t=0$ and $t=u^{r_\rho}(x,y)$. In either case, $\rho=0$ or $\rho=
u^{r_\rho}(x,y)$ and we conclude
$$
\rho\le u^{r_\rho}(x,y).
$$
\end{proof}
The main conclusion of the above lemma is that when studying the eradication times $u^{r_{\tau}}(x,y)$ we only need to consider values of $\tau$ for which $\tau\le u^{r_{\tau}}(x,y)$.
\begin{proof}[Proof of Theorem \ref{thm2}]
Let $S, I$ be the solution of the SIR system \eqref{SIR} with $S(0)=x\ge 0$ and $I(0)=y\ge \mu$.  According to \cite{MR3688684} and Theorem \ref{ExistenceTheorem}, there is $\tau\ge 0$ such that
\be\label{tauMinimizingFormula}
u(x,y)=\min_{r\in {\cal A}}u^{r}(x,y)=u^{r_{\tau}}(x,y).
\ee
By Lemma \ref{restricttaulessthanuarrtau}, we may assume $\tau\le u^{r_{\tau}}(x,y) =u(x,y)$. As a result, 
\be\label{tauMinimizingFormula2}
u(x,y)=\min_{0\le \tau\le u(x,y)}u^{r_\tau}(x,y).
\ee

\par Also note 
\begin{align}
u^{r_\tau}(x,y)&=\tau+u^{r_\tau}(S^{r_\tau}(\tau),I^{r_\tau}(\tau))\\
&=\tau+u^{r_0}(S(\tau),I(\tau))
\end{align}
for $\tau\le u(x,y)$. Here we used 
$$
\begin{cases}
(S^{r_\tau}(\tau),I^{r_\tau}(\tau))=(S(\tau),I(\tau)) \text{ and}\\
u^{r_\tau}(S^{r_\tau}(\tau),I^{r_\tau}(\tau))=u^{r_0}(S(\tau),I(\tau)),
\end{cases}
$$
which follow from the definitions of $r_\tau$ and $r_0$. 

\par In view of \eqref{tauMinimizingFormula2},
\be
u(x,y)=\min_{0\le \tau\le u(x,y)}\{\tau+u^{r_0}(S(\tau),I(\tau))\}.
\ee
Of course if $\tau>u(x,y)$, then $\tau+u^{r_0}(S(\tau),I(\tau))>u(x,y)$. Consequently, 
\be\label{tauMinimizingFormula}
u(x,y)=\min_{\tau\ge 0}\{\tau+u^{r_0}(S(\tau),I(\tau))\}.
\ee
In addition, we note that if $u(x,y)=\tau+u^{r_0}(S(\tau),I(\tau))$, then $\tau\le u(x,y)\le u^{r_\tau}(x,y)$ which in turn implies $u(x,y)=\tau+u^{r_0}(S(\tau),I(\tau))=u^{r_\tau}(x,y)$. 
\\
\par Now set 
\be\label{OptimalStoppingTimetau}
\tau^*:=\inf\{t\ge 0: u(S(t),I(t))=u^{r_0}(S(t),I(t)) \},
\ee
and choose a minimizing $\tau\ge 0$ in \eqref{tauMinimizingFormula}.  Then 
$$
u^{r_\tau}(x,y)=\tau+u^{r_0}(S(\tau),I(\tau))\ge \tau.
$$
We also have
$$
u(S(\tau),I(\tau))=u^{r_\tau}(S(\tau),I(\tau))
$$
by Proposition \ref{DPPprop}. In addition, we recall $u^{r_\tau}(S(\tau),I(\tau))=u^{r_0}(S(\tau),I(\tau))$ so that 
$$
\tau^*\le \tau<\infty.
$$

\par Moreover, we only need to consider times $t\ge 0$ in \eqref{OptimalStoppingTimetau} which are bounded above by $\tau$.  It now follows easily that the infimum in \eqref{OptimalStoppingTimetau} is attained by $t=\tau^*$. And appealing to Proposition \ref{DPPprop} once again gives
\begin{align}
u(x,y)&=u^{r_\tau}(x,y)\\
&=\tau^*+u^{r_\tau}(S(\tau^*),I(\tau^*))\\
&\ge \tau^*+u(S(\tau^*),I(\tau^*))\\
&= \tau^*+u^{r_0}(S(\tau^*),I(\tau^*)).
\end{align}
Thus $\tau^*$ is optimal. 
\end{proof}
We will make use of the fact that $u^{r_0}$ is a smooth solution of  
\be\label{urzeroPDE}
\beta xy \partial_xu^{r_0}+x\partial_xu^{r_0}+(\gamma-\beta x)y\partial_yu^{r_0}=1
\ee
in $(0,\infty)\times(\mu,\infty)$.  This follows from computing the time derivative of both sides of the identity 
$$
u^{r_0}(x,y)=t+u^{r_0}(S^{r_0}(t),I^{r_0}(t))
$$
at $t=0$; here $S^{r_0}, I^{r_0}$ is the solution of the \eqref{SIRcontrol} with $S^{r_0}(0)=x$ and $I^{r_0}(0)=y$. 
\begin{cor}\label{DerivativesurzeroCor}
Suppose $x>0$, $y>\mu$. If $u(x,y)=u^{r_0}(x,y)$, then 
$$
\partial_xu^{r_0}(x,y)\ge 0.
$$
Otherwise
\be\label{urzeroxequalzero}
\partial_xu^{r_0}(S(\tau^*),I(\tau^*))=0,
\ee
where $\tau^*$ is given by \eqref{OptimalStoppingTimetau} and $S, I$ is the solution of the SIR system \eqref{SIR} with $S(0)=x$ and $I(0)=y$.
\end{cor}
\begin{proof}
If $u(x,y)=u^{r_0}(x,y)$, then $\tau=0$ is a minimizing time in \eqref{tauMinimizingFormula}.  In view of \eqref{urzeroPDE}, 
\begin{align}
0&\le\left. \frac{d}{d\tau}\left(\tau+ u^{r_0}(S(\tau),I(\tau))\right)\right|_{\tau=0}\\
&=1-\beta xy \partial_xu^{r_0}(x,y)-(\gamma-\beta x)y\partial_yu^{r_0}(x,y)\\
&=x\partial_xu^{r_0}(x,y).
\end{align}
If $u(x,y)<u^{r_0}(x,y)$, then $\tau^*>0$, and we can perform a computation similar to the one above to find \eqref{urzeroxequalzero}.
\end{proof}
\par  Combining \eqref{tauMinimizingFormula} with the dynamic programming principle \eqref{DPP} gives 
\be\label{DPP2}
u(x,y)=\min_{\tau\ge 0}\{\tau+u(S(\tau),I(\tau))\},
\ee
where $S, I$ is the solution of the SIR system \eqref{SIR} with $S(0)=x$ and $I(0)=y\ge \mu$. We will use this identity to verify the following claim. 
\begin{cor}
The value function $u$ is a viscosity solution of the PDE \eqref{HJB2} 
\be
\max\{\beta xy \partial_xu+(\gamma-\beta x)y\partial_yu-1, u-u^{r_0}\}=0
\ee
in $(0,\infty)\times(\mu,\infty)$. 
\end{cor}
\begin{proof}
Let $x_0>0$ and $y_0>\mu$, and suppose $\varphi$ is continuously differentiable in a neighborhood of $(x_0,y_0)$ and that $u-\varphi$ has a local maximum at $(x_0,y_0)$.  Then
$$
(u-\varphi)(S(t),I(t))\le (u-\varphi)(x_0,y_0)
$$
for all $t\ge 0$ small. Here $S, I$ is the solution of the SIR system \eqref{SIR} with $S(0)=x_0$ and $I(0)=y_0$.
By \eqref{DPP2}, $u(x_0,y_0)\le t+ u(S(t),I(t))$ for all $t\ge 0$,  so that 
$$
-t\le u(S(t),I(t))-u(x_0,y_0)\le \varphi(S(t),I(t))-\varphi(x_0,y_0)
$$
for all $t\ge 0$ small. Consequently, 
\begin{align}
-1&\le \left.\frac{d}{dt} \varphi(S(t),I(t))\right|_{t=0}\\
&=-\beta x_0y_0 \partial_x\varphi(x_0,y_0)-(\gamma-\beta x_0)y_0\partial_y\varphi(x_0,y_0).
\end{align}
Therefore, 
$$
\max\{\beta x_0y_0 \partial_x\varphi(x_0,y_0)+(\gamma-\beta x_0)y_0\partial_y\varphi(x_0,y_0)-1, u(x_0,y_0)-u^{r_0}(x_0,y_0)\}\le 0.
$$

\par Now suppose $\psi$ is continuously differentiable and $u-\psi$ has a local minimum at $(x_0,y_0)$.  We claim 
\be\label{SuperSolnIneqOST}
\max\{\beta x_0y_0 \partial_x\psi(x_0,y_0)+(\gamma-\beta x_0)y_0\partial_y\psi(x_0,y_0)-1, u(x_0,y_0)-u^{r_0}(x_0,y_0)\}\ge 0.
\ee
Recall that if $u(x_0,y_0)<u^{r_0}(x_0,y_0)$, then the corresponding $\tau^*$ defined in \eqref{OptimalStoppingTimetau} is positive.  As a result, 
$$
u(x_0,y_0)=u^{r_{\tau^*}}(x_0,y_0)=t+u^{r_{\tau^*}}(S(t),I(t))=t+u(S(t),I(t))
$$
for $t\in [0, \tau^*]$. It follows that 
$$
-t= u(S(t),I(t))-u(x_0,y_0)\ge \psi(S(t),I(t))-\psi(x_0,y_0)
$$
for all $t>0$ small enough.  Therefore, 
\begin{align}
-1&\ge \left.\frac{d}{dt} \psi(S(t),I(t))\right|_{t=0}\\
&=-\beta x_0y_0 \partial_x\psi(x_0,y_0)-(\gamma-\beta x_0)y_0\partial_y\psi(x_0,y_0)
\end{align}
which implies \eqref{SuperSolnIneqOST}.
\end{proof}

\par We have established that the value function is a viscosity solution of the HJB \eqref{HJB} and the PDE \eqref{HJB2}. There is at least one implication of this fact 
which we can state in terms of the set ${\cal S}$ mentioned in our introduction
$$
{\cal S}=\{(x,y)\in(0,\infty)\times(\mu,\infty): u(x,y)=u^{r_0}(x,y) \}.
$$
\begin{cor}
For each $(x,y)$ belonging to the interior of  ${\cal S}$,
$$
\partial_xu(x,y)\ge 0.
$$
And at almost every $(x,y)\in {\cal S}^c$, 
$$
\partial_xu(x,y)\le 0.
$$
\end{cor}
\begin{proof}
\par As $u$ agrees with $u^{r_0}$ in ${\cal S}$, $u$ is smooth in the interior of ${\cal S}$. It follows from Corollary \ref{DerivativesurzeroCor} that $\partial_xu(x,y)
\ge0$ for each $(x,y)$ in the interior of ${\cal S}$.

\par Since $u$ is locally Lipschitz on $(0,\infty)\times(\mu,\infty)$, $u$ is differentiable almost everywhere in ${\cal S}^c$. Let $(x,y)\in {\cal S}^c$ be a differentiability point of $u$. As $u$ is a viscosity solution of the HJB \eqref{HJB}, it is routine to check that $u$ satisfies the equation at this point. That is, 
$$
\beta xy \partial_xu(x,y)+x\partial_xu(x,y)^++(\gamma-\beta x)y\partial_yu(x,y)=1.
$$  
See Proposition 1.9 of Chapter II in \cite{MR1484411}, and Corollary 8.1 of Chapter II in \cite{MR2179357} for more on this technical point. And since $u$ is a viscosity solution of \eqref{HJB2} and $(x,y)\in {\cal S}^c$, we also have
$$
\beta xy \partial_xu(x,y)+(\gamma-\beta x)y\partial_yu(x,y)=1.
$$
Upon subtracting these equations, we find $\partial_xu(x,y)^+=0$. That is,  $\partial_xu(x,y)\le 0$.
\end{proof}

\section{Necessary conditions revisited}\label{PMPsect}
In this final section, we will relate our ideas on dynamic programming back to the necessary conditions $(i)-(vi)$ which follow from Pontryagin's maximum principle.  The link between viscosity solutions of Hamilton-Jacobi equations and Pontryagin's maximum principle was first established by Barron and Jensen \cite{MR860384}.  Our particular control problem does not exactly fit into the framework they considered, so we cannot simply quote their results. Nevertheless, the ideas presented below are inspired by their work. 

\par Our first insight is that each optimal vaccination rate $r\in {\cal A}$ is a ``feedback" control. That is, $r(t)$ depends on the value of $(S^r(t),I^r(t))$ for almost every $t\ge 0$.  In proving this assertion, we  will make use of the following basic observation. Whenever $x_0> 0$, $y_0> \mu$, and $u(x_0,y_0)=u^r(x_0,y_0)$, then
$$
u(x,y)-u^r(x,y)\le 0=u(x_0,y_0)-u^r(x_0,y_0)
$$
for each $x>0$ and $y>\mu$.  That is, $u-u^r$ has a maximum at $(x_0,y_0)$. Since $u$ is a viscosity solution and $u^r$ is continuously differentiable, 
\be\label{LastSupsolnInequArr}
\beta x_0y_0 \partial_xu^r(x_0,y_0)+x_0(\partial_xu^r(x_0,y_0))^++(\gamma-\beta x_0)y_0\partial_yu^r(x_0,y_0)\le 1.
\ee

\begin{prop}\label{OptimalSynthesisProp}
Let $x> 0$ and $y>\mu$ and choose $r\in {\cal A}$ such that $u(x,y)=u^r(x,y).$ Then 
\be\label{OptimalConditionArr}
r(t)\partial_xu^r(S^r(t),I^r(t))=\partial_xu^r(S^r(t),I^r(t))^+
\ee
for almost every $t\in [0,u(x,y)]$ and 
\begin{align}\label{HJconditionuArr}
&\beta S^{r}(t)I^{r}(t)\partial_xu^r(S^{r}(t),I^{r}(t))+S^{r}(t)\partial_xu^r(S^{r}(t),I^{r}(t))^+\\
&\hspace{1.5in} +(\gamma-\beta S^{r}(t))I^{r}(t)\partial_yu^r(S^{r}(t),I^{r}(t))=1
\end{align}
for all $t\in [0,u(x,y)]$.
\end{prop}
\begin{proof}
By Proposition \ref{DPPprop},
$$
u(S^r(t),I^r(t))=u^{r}(S^{r}(t),I^{r}(t))
$$
for $t\in [0,u(x,y)]$. And in view of inequality \eqref{LastSupsolnInequArr},
\begin{align}\label{HJconditionuArr2}
&\beta S^r(t)I^r(t) \partial_xu^r(S^r(t),I^r(t))+S^r(t)(\partial_xu^r(S^r(t),I^r(t)))^+\\
&\hspace{1.5in}+(\gamma-\beta S^r(t))I^r(t)\partial_yu^r(S^r(t),I^r(t))\le 1
\end{align}
for all $t\in [0,u(x,y)]$.  Furthermore, we always have 
$$
u^{r}(S^{r}(t),I^{r}(t))=u^r(x,y)-t
$$
for $t\in [0,u(x,y)]$. Differentiating gives 
\begin{align}
-1&=\frac{d}{dt}u^{r}(S^{r}(t),I^{r}(t))\\
&=-\beta S^r(t)I^r(t) \partial_xu^r(S^r(t),I^r(t))-S^r(t)r(t)\partial_xu^r(S^r(t),I^r(t))\\
&\hspace{1.5in} -(\gamma-\beta S^r(t))I^r(t)\partial_yu^r(S^r(t),I^r(t))\\
&\ge -\beta S^r(t)I^r(t) \partial_xu^r(S^r(t),I^r(t))-S^r(t)(\partial_xu^r(S^r(t),I^r(t)))^+\\
&\hspace{1.5in}-(\gamma-\beta S^r(t))I^r(t)\partial_yu^r(S^r(t),I^r(t))\\
&\ge -1
\end{align}
for almost every $t\in [0,u(x,y)]$; the last inequality is due to \eqref{HJconditionuArr2}. We conclude \eqref{OptimalConditionArr} and \eqref{HJconditionuArr} hold 
for almost every $t\in [0,u(x,y)]$.  Since $\partial_xu^r$ and $\partial_yu^r$ are continuous, \eqref{HJconditionuArr} actually holds for all $t\in [0,u(x,y)]$.
\end{proof}

\par We will need to record a basic fact involving the adjoint equations appearing in the necessary conditions obtained via Pontryagin's maximum principle. 
\begin{lem}\label{AdjointLemma}
Let $r\in {\cal A}$, $x>0$, and $y>\mu$.  Set 
$$
P(t)=\partial_xu^r(S^r(t),I^r(t))\quad \text{and}\quad Q(t)=\partial_yu^r(S^r(t),I^r(t))
$$
where $S^r$ and $I^r$ is the solution of \eqref{SIRcontrol} with $S^r(0)=x$ and $I^r(0)=y$. Then $P, Q$ satisfy
\be\label{HODE}
\begin{cases}
\dot P(t) = (\beta I(t)+r(t))P(t)-\beta I(t)Q(t)\\
\dot Q(t) = \beta S(t)P(t)+(\gamma-\beta S(t))Q(t)
\end{cases}
\ee
for almost every $t\in [0,u^r(x,y)]$.
\end{lem} 
\begin{proof}
Just as we computed \eqref{linearizationfirstinx}, we have 
\be\label{linearizationfirstinxandy}
\partial_tZ(x,y,t)
=A(x,y,t)Z(x,y,t)
\ee
for almost every $t\in [0,u^r(x,y)]$ where 
$$
Z(x,y,t)=\left(\begin{array}{cc}
 \partial_x\Phi^r_1(x,y,t)& \partial_y\Phi^r_1(x,y,t)  \\
 \partial_x\Phi^r_2(x,y,t)  & \partial_y\Phi^r_2(x,y,t)  \\
\end{array}\right)
$$
and
$$
A(x,y,t)=\left(\begin{array}{cc}
-\beta \Phi^r_2(x,y,t) -r(t) & -\beta \Phi^r_1(x,y,t) \\
\beta \Phi^r_2(x,y,t) & \beta \Phi^r_1(x,y,t) -\gamma
\end{array}\right).
$$
Taking the transpose of \eqref{linearizationfirstinxandy} leads to
\be\label{FirstZdotidentity} 
\partial_tZ(x,y,t)^t
=Z(x,y,t)^tA(x,y,t)^t.
\ee
We also note that since 
$$
Z(x,y,0)=\left(\begin{array}{cc}
1& 0  \\
0  & 1 \\
\end{array}\right),
$$
$t\mapsto Z(x,y,t)$ is the fundamental solution of the $2\times 2$ system \eqref{linearizationfirstinxandy}. In particular,  $Z(x,y,t)$ is a nonsingular matrix for each $t\ge 0$.

\par Recall the identity
$$
u^r( \Phi^r_1(x,y,t), \Phi^r_2(x,y,t))=u^r(x,y)-t
$$
for $t\in [0,u^r(x,y)]$. Differentiating with respect to $x$ and $y$ gives
$$
Z(x,y,t)^t
\left(\begin{array}{c}
\partial_xu^r( \Phi^r_1(x,y,t), \Phi^r_2(x,y,t))\\
\partial_yu^r( \Phi^r_1(x,y,t), \Phi^r_2(x,y,t))
\end{array}\right)=\left(\begin{array}{c}
\partial_xu^r(x,y)\\
\partial_yu^r(x,y)
\end{array}\right).
$$
And taking the derivative with respect to $t$ leads to
\begin{align}\label{SecondZdotIdentity}
\partial_tZ(x,y,t)^t
\left(\begin{array}{c}
\partial_xu^r( \Phi^r_1(x,y,t), \Phi^r_2(x,y,t))\\
\partial_yu^r( \Phi^r_1(x,y,t), \Phi^r_2(x,y,t))
\end{array}\right)+ Z(x,y,t)^t  \partial_t
\left(\begin{array}{c}
\partial_xu^r( \Phi^r_1(x,y,t), \Phi^r_2(x,y,t))\\
\partial_yu^r( \Phi^r_1(x,y,t), \Phi^r_2(x,y,t))
\end{array}\right)=0.
\end{align}

\par Let us now fix $x>0$ and $y>\mu$ and set $Z(t)=Z(x,y,t)$, $A(t)=A(x,y,t)$, and 
$$
\left(\begin{array}{c}
P(t)\\
Q(t)
\end{array}\right)
=\left(\begin{array}{c}
\partial_xu^r( \Phi^r_1(x,y,t), \Phi^r_2(x,y,t))\\
\partial_yu^r( \Phi^r_1(x,y,t), \Phi^r_2(x,y,t))
\end{array}\right).
$$
By \eqref{FirstZdotidentity} and \eqref{SecondZdotIdentity},
$$
Z(t)^t\frac{d}{dt}\left(\begin{array}{c}
P(t)\\
Q(t)
\end{array}\right)=-Z(t)^tA(t)^t\left(\begin{array}{c}
P(t)\\
Q(t)
\end{array}\right).
$$
for almost every $t\in [0,u^r(x,y)]$. Since $Z(t)$ is nonsingular, 
$$
\frac{d}{dt}\left(\begin{array}{c}
P(t)\\
Q(t)
\end{array}\right)=-A(t)^t\left(\begin{array}{c}
P(t)\\
Q(t)
\end{array}\right)
$$
which is \eqref{HODE}.
\end{proof}

\par We can now establish the necessary conditions coming from Pontryagin's maximum principle in terms of the derivatives of $u^r$ when $u(x,y)=u^r(x,y)$. 
\begin{proof}[Proof of Theorem \ref{thm3}]
Properties $(iii)$ and $(iv)$ were established in Proposition \ref{OptimalSynthesisProp}. As for property $(ii)$, recall that
$\partial_xu^r(x,\mu)=0$ for $x\in (0,\gamma/\beta)$ as explained in Remark \ref{Extenduxuy}.  In view of Corollary \ref{Sulessthan}, $S^r(u)\in (0,\gamma/\beta)$ so  
$P(u)=0$.  Moreover, evaluating \eqref{HJconditionuArr} at $t=u$ gives
$$
(\gamma-\beta S^{r}(u))\mu Q(u)=1.
$$
Thus, $Q(u)\neq 0$.  Finally, property $(i)$ follows from Lemma \ref{AdjointLemma}. 
\end{proof}

\noindent {\bf Acknowledgements}: This material is based upon work supported by the National Science Foundation
under Grants No. DMS-1440140 and DMS-1554130, National Security Agency under Grant No.
H98230-20-1-0015, and the Sloan Foundation under Grant No. G-2020-12602 while the
authors participated in a program hosted by the Mathematical Sciences Research Institute in Berkeley, California, during the summer of 2020.

\bibliography{SIRbib}{}

\bibliographystyle{plain}

\typeout{get arXiv to do 4 passes: Label(s) may have changed. Rerun}

\end{document}